\documentclass[preprint,hidelinks,onefignum,onetabnum,reqno]{siamart250211}
\usepackage{lipsum}
\usepackage{amsfonts}
\usepackage{graphicx}
\usepackage{epstopdf}
\usepackage{algorithmic}
\usepackage{subfig}
\usepackage{cleveref}
\usepackage{hyperref}

\ifpdf
  \DeclareGraphicsExtensions{.eps,.pdf,.png,.jpg}
\else
  \DeclareGraphicsExtensions{.eps}
\fi

\newsiamremark{remark}{Remark}
\newsiamremark{hypothesis}{Hypothesis}
\crefname{hypothesis}{Hypothesis}{Hypotheses}
\newsiamthm{claim}{Claim}
\newsiamremark{fact}{Fact}
\crefname{fact}{Fact}{Facts}

\headers{A DeepLagrangian method for Keller-Segel chemotaxis systems }{Y. Feng, M. K. Ng and Z. Zhang}

\title{A DeepLagrangian method for learning and generating aggregation patterns in multi-dimensional Keller-Segel chemotaxis systems\thanks{Submitted to the editors DATE.
\funding{M. Ng’s research is supported in part by the National Key Research and Development Program of China under Grant 2024YFE0202900, GDSTC: Guangdong and Hong Kong Universities “1+1+1” Joint Research Collaboration Scheme UICR0800008-24, Hong Kong RGC grant project (No. 12300125), and Joint NSFC and RGC N-HKU769/21. Z. Zhang’s research is supported in part by the National Natural Science Foundation of China (No. 92470103), Hong Kong RGC grant project (No. 17300325), and Seed Funding for Strategic Interdisciplinary Research Scheme 2021/22 (HKU). }}}
\author{Yani Feng\thanks{Department of Mathematics, The University of Hong Kong
  (\email{fengyn@hku.hk}, \email{zhangzw@hku.hk}).}
\and Michael K. Ng\thanks{Department of Mathematics, Hong Kong Bapist University
  (\email{michael-ng@hkbu.edu.hk}).}
\and Zhiwen Zhang\footnotemark[2]}

\usepackage{amsopn}

\ifpdf
\hypersetup{
  pdftitle={A DeepLagrangian method for learning and generating aggregation patterns in multi-dimensional Keller-Segel chemotaxis systems},
  pdfauthor={Y. Feng, M. K. Ng and Z. Zhang}
}
\fi

\begin{document}

\maketitle

% REQUIRED
\begin{abstract}
The Keller-Segel (KS) chemotaxis system is used to describe the overall behavior of a collection of cells under the influence of chemotaxis. However, solving the KS chemotaxis system and generating its aggregation patterns remain challenging due to the emergence of solutions exhibiting near-singular behavior, such as finite-time blow-up or concentration phenomena. Building on a Lagrangian framework of the KS system, we develop DeepLagrangian, a self-adaptive density estimation method that learns and generates aggregation patterns and near-singular solutions of the KS system in two- and three-dimensional (2D and 3D) space under different physical parameters. The main advantage of the Lagrangian framework is its inherent ability to adapt to near-singular solutions. To develop this framework, we normalize the KS solution into a probability density function (PDF), derive the corresponding normalized KS system, and utilize the property of the continuity equation to rewrite the system into a Lagrangian framework. We then define a physics-informed Lagrangian loss to enforce this framework and incorporate a flow-based generative model, called the time-dependent KRnet, to approximate the PDF by minimizing the loss. Furthermore, we integrate time-marching strategies with the time-dependent KRnet to enhance the accuracy of the PDF approximation.  After obtaining the approximate PDF, we recover the original KS solution. We also prove that the Lagrangian loss effectively controls the Kullback-Leibler (KL) divergence between the approximate PDF and the exact PDF. In the numerical experiments, we demonstrate the accuracy of our DeepLagrangian method for the 2D and 3D KS chemotaxis system with/without advection.
\end{abstract}

% REQUIRED

% REQUIRED

\section{Introduction}
Chemotaxis is a fundamental biological phenomenon where organisms or cells move directionally in response to spatial gradients of chemical substances. 
This behavior underlies diverse behaviors, from microbial survival to complex multicellular functions, with profound implications for health and disease \cite{zhou2023bacterial,roussos2011chemotaxis,cremer2019chemotaxis}.
Chemotaxis is often modeled using partial differential equations (PDEs) to describe the interplay between cell movement and the dynamics of chemical signals.
The Keller-Segel (KS) system \cite{keller1970initiation} is the first mathematical model for chemotaxis to describe the aggregation behavior of the cellular slime mold Dictyostelium discoideum.
The general form of the KS system is 
\begin{equation}\label{ks}
	\left\{
	\begin{array}{ll}
	\frac{\partial \rho(\mathbf{x},t)}{\partial t}=\nabla\cdot(\mu \nabla \rho(\mathbf{x},t)-\chi \rho(\mathbf{x},t) \nabla c(\mathbf{x},t)), &\mathbf{x} \in\Omega,t>0,\\
	\epsilon \frac{\partial c(\mathbf{x},t)}{\partial t}=\Delta c(\mathbf{x},t)-k^2 c(\mathbf{x},t)+\rho(\mathbf{x},t), &\mathbf{x}\in\Omega,t>0,\\
	\frac{\partial \rho(\mathbf{x},t)}{\partial \nu}=\frac{\partial c(\mathbf{x},t)}{\partial \nu}=0, &\mathbf{x}\in\partial\Omega, t>0,\\
\rho(\mathbf{x},0)=\rho^0(\mathbf{x}),c(\mathbf{x},0)=c^0(\mathbf{x}),&\mathbf{x}\in \Omega,
	\end{array}
	\right.
\end{equation}
in a bounded domain $\Omega\subset \mathbb{R}^d$ with a smooth boundary $\partial \Omega$, where $\frac{\partial}{\partial \nu}$ denotes differentiate with respect to the outward normal $\nu$ on $\partial \Omega$. The initial conditions $\rho^0(\mathbf{x})$ and $c^0(\mathbf{x})$ are assumed to be nonnegative.
$\rho(\mathbf{x},t)$ and $c(\mathbf{x},t)$ represent the density of the cell population and the chemical concentration, respectively. $\mu>0$ is the cellular motility. The constant $\chi>0$ measures the attraction rate of the chemical gradient on the directed cellular movement. $k^2$ models the phenomenon that the chemical is consumed by certain enzymes in the environment.

In the special case where $\epsilon=0$ and $k=0$, the second equation of system \eqref{ks} reduces to the classical Poisson equation
$-\Delta c=\rho(\mathbf{x},t)$. The chemical concentration can be expressed as 
\begin{align}\label{poisson}
    c(\mathbf{x},t)=-\mathcal{K}*\rho.
\end{align}
Here $\mathcal{K}$ denotes the Green function of the Laplacian operator and $*$ means convolution, e.g. $d=2$,
$
    \mathcal{K}(\mathbf{x},\mathbf{y})=\frac{1}{2\pi}\log {\| \mathbf{x}-\mathbf{y} \|},
$
for $d=3$,
$
    \mathcal{K}(\mathbf{x},\mathbf{y})=-\frac{1}{4\pi \| \mathbf{x}-\mathbf{y} \|}.
$
Substituting \eqref{poisson} into the first equation of \eqref{ks} gives the following non-linear advection–diffusion PDE,
\begin{align}\label{dens}
    \frac{\partial \rho(\mathbf{x},t)}{\partial t}=\nabla\cdot(\mu \nabla \rho(\mathbf{x},t)+\chi \rho(\mathbf{x},t) \nabla (\mathcal{K}*\rho)), \mathbf{x}\in \Omega,t>0.
\end{align}
One can add the advective Lie derivative $\rho(\mathbf{x},t)$ on the left hand side of \eqref{dens} to describe chemotaxis in a fluid environment \cite{kiselev2016suppression,khan2015global}:
\begin{align}\label{simple_ks}
    \frac{\partial \rho(\mathbf{x},t)}{\partial t}+\nabla \cdot(\rho(\mathbf{x},t) \mathbf{v})=\nabla\cdot(\mu \nabla \rho(\mathbf{x},t)+\chi \rho(\mathbf{x},t) \nabla (\mathcal{K}*\rho)), \mathbf{x}\in \Omega,t>0.
\end{align}
The flow field $\mathbf{v}$ is generally known with physical parameters and 
can be applied to alleviate the blow-up or aggregation behavior of \eqref{dens}. 

Traditional numerical methods for the KS system include finite element methods, finite volume methods, and spectral methods. Carrillo et al. \cite{carrillo2019hybrid} propose a hybrid mass transport finite element method for 1D KS-type systems. Filbet \cite{filbet2006finite} develops a finite volume scheme
for the 2D Patlak–Keller–Segel
chemotaxis model. Shen et al. \cite{shen2020unconditionally} use the Fourier spectral method for a class of 2D KS equations. Strehl et al. \cite{strehl2013positivity} give a positivity-preserving finite element method for 3D chemotaxis problems. These methods need to discretize the KS system on mesh grids, but the number of mesh grid points scales exponentially with the spatial dimension.

In recent years, the advent of deep neural networks (DNNs) has opened up new possibilities for solving PDEs, primarily by representing a PDE solution through neural network parameterization. Many efficient approaches include Physics Informed Neural Networks (PINNs) \cite{raissi2019physics},
Deep Galerkin Method (DGM) \cite{sirignano2018dgm} and  Deep Ritz Method (DRM) \cite{yu2018deep}.  
However, they encounter significant difficulties in solving PDEs with near-singular solutions, such as the KS system. There exist some works addressing the issue, including singularity-enriched PINN (SEPINN) \cite{hu2024solving} and DAS-PINNs \cite{tang2023pinns}. SEPINN enriches the ansatz space spanned by deep neural networks by incorporating suitable singular functions to obtain more accurate solutions. DAS-PINNs develop a deep adaptive sampling method and combine it with PINNs to improve the accuracy of PDE solutions iteratively. These methodologies predominantly occur within the Eulerian framework, where PDE solutions are defined either in the strong sense or the weak (variational) sense.

To address the issue, we develop a Lagrangian framework for computing aggregation patterns and near-singular solutions of the KS chemotaxis system \eqref{simple_ks}, referred to as the DeepLagrangian method, which is a self-adaptive density estimation method. The fundamental concept of the Lagrangian framework is to represent PDE solutions using particles. It has three key advantages: first, it is mesh-free in space; second, it is self-adaptive, enabling it to naturally track near-singular solutions; and third, its computational costs increase linearly with the dimension of the spatial variables. Some progress has been made in developing Lagrangian methods to solve a class of KS chemotaxis systems \cite{liu2017random,havskovec2009stochastic,wang2025novel} and compute effective diffusivities in chaotic or random flows \cite{wang2018computing,lyu2020convergence}. Recently, deep learning-based Lagrangian methods have been proposed to solve time-dependent Fokker-Planck equations \cite{shen2022self,li2023self,boffi2023probability} and, more generally, McKean-Vlasov Type PDEs \cite{shen2023entropy}.  Our DeepLagrangian method falls within this category of deep learning-based Lagrangian methods.
 
To establish the Lagrangian framework, the first step of our DeepLagrangian method is to normalize a PDE solution of \eqref{simple_ks} into a PDF from which a new KS equation for the PDF can be derived correspondingly. Then, we reformulate the new KS equation into a Lagrangian framework by using the property of the continuity equation.
Note that the Lagrangian framework involves both the PDF and the corresponding samples/particles. However, directly approximating the PDF and generating the samples/particles using conventional neural networks (e.g., Multi-Layer Perceptrons (MLPs), Convolutional Neural Networks (CNNs)) causes some difficulties: first, conventional neural networks cannot directly guarantee the non-negativity, normalization, and vanishing at infinity properties of the PDF; second, drawing samples/particles from a PDF parameterized by such networks is intractable. Flow-based generative models, such as NICE \cite{Dinh_2014}, Real NVP \cite{dinh2016density}, and KRnet \cite{tang2020deep}, are capable of overcoming these difficulties. 
In particular, we employ a time-dependent KRnet to approximate the PDF, as the KRnet offers greater expressive power than the Real NVP and tends to be more stable during training, especially when dealing with the KS system with near-singular solutions.

The main contributions of this paper are as follows:
\begin{itemize}
    \item We build a Lagrangian framework of the KS system and then develop DeepLagrangian for learning aggregation patterns and near-singular solutions of the KS system, where the KS solution is normalized into a PDF.
    \item We define a physics-informed Lagrangian loss \eqref{La} to enforce this framework and utilize a time-dependent KRnet to approximate the PDF by minimizing the loss, and incorporate a time-marching strategy with the time-dependent KRnet to enhance the accuracy of the PDF approximation.
    \item We provide a rigorous theoretical analysis to show that the KL divergence between the exact PDF and the approximate PDF can be bounded by the Lagrangian loss (see \Cref{therorem 4}).
    \item We demonstrate the accuracy of our DeepLagrangian method through extensive numerical experiments across various scenarios, including 2D KS systems with different initial conditions and physical parameters, as well as 3D KS systems. Our results also show that our method outperforms Eulerian methods such as PINNs (see \Cref{test1_section}).
\end{itemize}

The rest of the paper is organized as follows. In Section 2, we present a detailed derivation of the Lagrangian framework of the KS equation, where the solution of the KS equation is normalized into a PDF. In Section 3, we propose our DeepLagrangian method: we define a physics-informed Lagrangian loss and introduce a time-dependent KRnet, combined with time marching strategies, to approximate the PDF by minimizing the loss. In Section 4, we prove that the Lagrangian loss can control the KL divergence between the exact PDF and the approximate PDF. In Section 5, we conduct numerical experiments to demonstrate the accuracy of our DeepLagrangian method. Finally, some concluding remarks are given in Section 6.

\section{Lagrangian framework}
The KS equation mentioned in Section 1 is considered, 
\begin{align}\label{fks}
    \frac{\partial \rho(\mathbf{x},t)}{\partial t}+\nabla \cdot(\rho(\mathbf{x},t) \mathbf{v})=\nabla\cdot(\mu \nabla \rho(\mathbf{x},t)+\chi \rho(\mathbf{x},t) \nabla (\mathcal{K}*\rho)), \mathbf{x}\in \Omega,t>0,
\end{align}
with the initial condition $\rho^0(\mathbf{x})$ and $\Omega\subset \mathbb{R}^d$.
The key property of \eqref{fks} is the conservation of mass: 
\begin{align*}
    \frac{d}{dt}\int_{\Omega}\rho(\mathbf{x},t)d\mathbf{x}&=-\int_{\Omega}\nabla \cdot(\rho(\mathbf{x},t) \mathbf{v})d\mathbf{x}+\int_{\Omega}\nabla\cdot(\mu \nabla \rho(\mathbf{x},t)+\chi \rho(\mathbf{x},t) \nabla (\mathcal{K}*\rho))d\mathbf{x}\\
    &=-\int_{\partial \Omega} \rho(\mathbf{x},t) \mathbf{v}\cdot \nu d\mathbf{x}+\int_{\partial \Omega} (\mu \nabla \rho(\mathbf{x},t)-\chi \rho(\mathbf{x},t) \nabla c)\cdot \nu d\mathbf{x}\\
    &=0,
\end{align*}
where $\rho(\mathbf{x},t) \to 0$ as $\Vert \mathbf{x} \Vert \to 0$.
It follows that the total mass $M:=\int_{\Omega}\rho(\mathbf{x},t)d \mathbf{x}$ is equal to $\int_{\Omega}\rho^0(\mathbf{x}) d\mathbf{x}$. For the KS equation \eqref{fks} without advection ($d=2$,$\chi=\mu=1$, $\mathbf{v}=0$), it is known that if the total mass $M$ is greater than $8\pi$, the system will blow up in a finite time \cite{perthame2004pde}.
 Our interest lies in studying near-singular solutions of such KS equations using deep learning-based Lagrangian methods. 

To derive a Lagrangian framework, we first normalize the time-dependent density $\rho(\mathbf{x},t)$ in \eqref{fks} by the total mass $M$, defining the probability density function $\bar{\rho}(\mathbf{x},t)=\rho(\mathbf{x},t)/M$. 
Substituting this normalization into the original KS equation \eqref{fks} yields the following form: 
\begin{align}\label{fks_density}
    \frac{\partial \bar{\rho}(\mathbf{x},t)}{\partial t}+\nabla \cdot\Big(\bar{\rho}(\mathbf{x},t) \mathbf{v}\Big)=\nabla\cdot\Big(\mu \nabla \bar{\rho}(\mathbf{x},t)+\chi M\bar{\rho}(\mathbf{x},t) \nabla (\mathcal{K}*\bar{\rho})\Big),
\end{align}
with the initial PDF $\bar{\rho}^0(\mathbf{x}):=\rho^0(\mathbf{x})/M$.
Furthermore, \eqref{fks_density} can be reformulated as
\begin{align*}
    \frac{\partial \bar{\rho}(\mathbf{x},t)}{\partial t}&=-\nabla \cdot\Big(\bar{\rho}(\mathbf{x},t) \mathbf{v}\Big)+\nabla\cdot\Big(\mu \bar{\rho}(\mathbf{x},t)\nabla \log\bar{\rho}(\mathbf{x},t)+\chi M\bar{\rho}(\mathbf{x},t) \nabla (\mathcal{K}*\bar{\rho})\Big),
\end{align*}
which is equivalent to 
\begin{align}
    &\frac{\partial \bar{\rho}(\mathbf{x},t)}{\partial t}+\nabla \cdot\Big(\bar{\rho}(\mathbf{x},t)v(\mathbf{x},t)\Big) =0, \label{ct_e}\\
    &v(\mathbf{x},t):=\mathcal{A}[\bar{\rho}(\mathbf{x},t)]=\mathbf{v}-\mu\nabla \log\bar{\rho}(\mathbf{x},t)-\chi M \nabla (\mathcal{K}*\bar{\rho}), \label{vv}
\end{align}
where $v$ is the underlying velocity field, and $\mathcal{A}[\bar{\rho}(\mathbf{x},t)]$ denotes the operator $\mathcal{A}$ acting on $\bar{\rho}(\mathbf{x},t)$.

It is important to see that \eqref{ct_e}
 corresponds to the diffusion-free Fokker-Planck equation, which can also be interpreted as the continuity equation. Mathematically, this equation is equivalent to the deterministic ordinary differential equation (ODE), 
\begin{equation}\label{ode}
\left\{
\begin{array}{ll}
\frac{d\mathbf{x}(t)}{dt}=v(\mathbf{x}(t),t), &t>0,\\
\mathbf{x}(0)\sim \bar{\rho}^0(\mathbf{x}).&\\
\end{array}
\right.
\end{equation}
By combining \eqref{vv} and \eqref{ode}, 
the Lagrangian framework of the normalized KS equation \eqref{fks_density} is given by, 
\begin{align}\label{particle_NN}
   \frac{d\mathbf{x}}{dt}=\mathbf{v}-\mu\nabla \log\bar{\rho}(\mathbf{x},t)-\chi M \nabla (\mathcal{K}*\bar{\rho}),\ \mathbf{x}(0) \sim \bar{\rho}^0(\mathbf{x}).
\end{align}

\section{Methodology} 
In this section, we present DeepLagrangian, a self-adaptive density estimation method. Initially, a physics-informed Lagrangian loss is defined. Subsequently, a flow-based generative model, i.e., the time-dependent KRnet, is employed to approximate the PDF $\bar{\rho}$ in \eqref{particle_NN} by minimizing the defined loss. Additionally, a time marching strategy is utilized to effectively capture  rapid variations of the time-dependent PDF $\bar{\rho}$ over time.

\subsection{Physics-informed Lagrangian loss}
Assume that $z$ follows a simple distribution $p_{\mathbf{z}}(\mathbf{z})$, such as a standard normal distribution. An invertible network with parameters $\theta$, denoted by $\Phi_{\theta}(\mathbf{z},t)$, is applied to parameterize $\mathbf{x}$ in \eqref{particle_NN}, i.e., 
\begin{align}\label{iNN}
    \mathbf{x}=\Phi_{\theta}(\mathbf{z},t), \quad \mathbf{z}=\Phi^{-1}_{\theta}(\mathbf{x},t),
\end{align}
and the resulting PDF $\bar{\rho}_{\theta}(\mathbf{x},t)$ can be estimated by the change of variable,
\begin{align}\label{pdf_compute}
    \bar{\rho}_{\theta}(\mathbf{x},t)=p_z(\Phi^{-1}_{\theta}(\mathbf{x},t)) \det\vert \nabla_x \Phi^{-1}_{\theta}(\mathbf{x},t)\vert.
\end{align} 
The velocity field $v_{\theta}(\mathbf{x},t)$ can be recovered via
\begin{align*}
	v_{\theta}(\mathbf{x},t)=\frac{d\mathbf{x}}{dt}=\frac{d \Phi_{\theta}(\mathbf{z},t)}{dt}=\frac{d \Phi_{\theta}(\Phi^{-1}_{\theta}(\mathbf{x},t),t)}{dt},
\end{align*}
which corresponds to 
\begin{align}
&\frac{\partial\bar{\rho}_{\theta}(\mathbf{x},t) }{\partial t}+\nabla \cdot(\bar{\rho}_{\theta}(\mathbf{x},t)v_{\theta}(\mathbf{x},t)) =0.\label{fks_approx}
\end{align}
\eqref{fks_approx} can be rewritten as
\begin{align}\label{fksA_density}
	\frac{\partial\bar{\rho}_{\theta}(\mathbf{x},t) }{\partial t}+\nabla \cdot\Big(\bar{\rho}_{\theta}(\mathbf{x},t)(\mathcal{A}[\bar{\rho}_{\theta}(\mathbf{x},t)]+\epsilon(\mathbf{x},t))\Big),
\end{align}
where $\epsilon(\mathbf{x},t)=v_{\theta}(\mathbf{x},t)-\mathcal{A}[\bar{\rho}_{\theta}(\mathbf{x},t)]$.

According to \eqref{particle_NN}, a physics-informed  Lagrangian loss is defined as follows, 
\begin{align}
    &\mathcal{L}(\theta)=\int_{0}^T\int_{\Omega} \Big\Vert \epsilon(\mathbf{x},t) \Big\Vert^2 \bar{\rho}_{\theta}(\mathbf{x},t) d\mathbf{x} dt \nonumber \\
	&=\int_{0}^T\int_{\Omega} \Big\Vert \frac{d\mathbf{x}}{dt}-\Big(\mathbf{v}-\mu\nabla \log\bar{\rho}_{\theta}(\mathbf{x},t)-\chi M \nabla (\mathcal{K}*\bar{\rho}_{\theta})\Big) \Big\Vert^2 \bar{\rho}_{\theta}(\mathbf{x},t) d\mathbf{x} dt.\label{La}
\end{align}
To compute $\mathcal{L}(\theta)$ via Monte Carlo integration, we begin by sampling $t_i \sim U[0,T]$, for $i=1,\dots,I$, 
    \begin{align*}
	&\mathcal{L}(\theta)\approx \frac{T}{I}\sum_{i=1}^I \int_{\Omega} \Big\Vert \frac{d\mathbf{x}}{dt}-\Big(\mathbf{v}-\mu\nabla \log\bar{\rho}_{\theta}(\mathbf{x},t_i)-\chi M \nabla (\mathcal{K}*\bar{\rho}_{\theta})\Big) \Big\Vert^2 \bar{\rho}_{\theta}(\mathbf{x},t_i) d\mathbf{x}\nonumber\\
    &\approx \frac{T}{I}\sum_{i=1}^I \int_{\Omega} \Big\Vert \frac{d\mathbf{x}}{dt}-\Big(\mathbf{v}-\mu\nabla \log\bar{\rho}_{\theta}(\mathbf{x},t_i)-\chi M \nabla (\mathcal{K}*\bar{\rho}_{\theta})\Big) \Big\Vert^2 p_{\mathbf{z}}(\mathbf{z}) dz,\, \mathbf{x}=\Phi_{\theta}(\mathbf{z},t_i),
    \end{align*}
    and then, for each $t_i$, we sample $\mathbf{x}^{i,j}\sim \bar{\rho}_{\theta}(\mathbf{x},t_i)$, i.e., $\mathbf{x}^{i,j}=\Phi_{\theta}(\mathbf{z}^{i,j},t_i)$, $\mathbf{z}^{i,j}\sim p_{\mathbf{z}}(\mathbf{z})$, for $j=1,\dots,J$,
    \begin{align}
	&\mathcal{L}(\theta)\approx \frac{T}{I}\frac{1}{J}\sum_{i=1}^I\sum_{j=1}^J \Big\Vert \frac{d\mathbf{x}^{i,j}}{dt}-\Big(\mathbf{v}-\mu\nabla \log\bar{\rho}_{\theta}(\mathbf{x}^{i,j},t_i)-\chi M \nabla (\mathcal{K}*\bar{\rho}_{\theta})\Big) \Big\Vert^2  \nonumber\\
	&\approx \frac{T}{I}\frac{1}{J}\sum_{i=1}^I\sum_{j=1}^J \Big\Vert \frac{d\mathbf{x}^{i,j}}{dt}-\Big(\mathbf{v}-\mu\nabla \log\bar{\rho}_{\theta}(\mathbf{x}^{i,j},t_i)-\chi M \int_{\Omega}\nabla \mathcal{K}(\mathbf{x}^{i,j},\mathbf{y})\bar{\rho}_{\theta}(\mathbf{y},t_i) d\mathbf{y} \Big) \Big\Vert^2 \nonumber\\
	&\approx\frac{T}{I}\frac{1}{J}\sum_{i=1}^I\sum_{j=1}^J \Big\Vert \frac{d\mathbf{x}^{i,j}}{dt}-\Big(\mathbf{v}-\mu\nabla \log\bar{\rho}_{\theta}(\mathbf{x}^{i,j},t_i)-\chi M \frac{1}{K}\sum_{k=1}^K \nabla \mathcal{K}(\mathbf{x}^{i,j},\mathbf{y}^{i,k}) \Big) \Big\Vert^2 \nonumber\\
	&\approx\frac{T}{I}\frac{1}{J}\sum_{i=1}^I\sum_{j=1}^J \Big\Vert \frac{d\mathbf{x}^{i,j}}{dt}-\Big(\mathbf{v}-\mu\nabla \log\bar{\rho}_{\theta}(\mathbf{x}^{i,j},t_i)-\chi M \frac{1}{K}\sum_{k=1}^K \nabla \mathcal{K}_{\delta}(\mathbf{x}^{i,j},\mathbf{y}^{i,k}) \Big) \Big\Vert^2,  \label{La_em}
\end{align}
where $\mathbf{y}^{i,k}=\Phi_{\theta}(\mathbf{z}^{i,k},t_i)$, $\mathbf{z}^{i,k}\sim p_{\mathbf{z}}(\mathbf{z})$, for $k=1,\dots,K$, and we set $K=J$ in numerical experiments.
The chemo-attract term 
$\nabla (\mathcal{K}*\bar{\rho}_{\theta}(\mathbf{x},t))$ is approximated via Monte-Carlo integration. Since the chemo-attract term 
$\nabla (\mathcal{K}*\bar{\rho}(\mathbf{x},t))$ can cause numerical instability when particles are very close together,
$\mathcal{K}$ is substituted with a smoothed approximation $\mathcal{K}_{\delta}$ \cite{wang2024deepparticle}, where $\delta$ is a regularization parameter.
 
In addition, \eqref{particle_NN} needs to satisfy the initial condition $\bar{\rho}^0(\mathbf{x})$.
This condition can either be inherently integrated into the network architecture \cite{he2025adaptive} or explicitly enforced by introducing an additional loss term $D_{KL}\left(\bar{\rho}^0(\mathbf{x})||\bar{\rho}_{\theta}(\mathbf{x},0)\right)$,
\begin{align*}
	D_{KL}\left(\bar{\rho}^0(\mathbf{x})||\bar{\rho}_{\theta}(\mathbf{x},0)\right)&=\int_{\Omega} \bar{\rho}^0(\mathbf{x})\log \frac{\bar{\rho}^0(\mathbf{x})}{\bar{\rho}_{\theta}(\mathbf{x},0)} d\mathbf{x} \nonumber\\
	&=\int_{\Omega} \bar{\rho}^0(\mathbf{x}) \log \bar{\rho}^0(\mathbf{x})d\mathbf{x}-\int_{\Omega} \bar{\rho}^0(\mathbf{x}) \log \bar{\rho}_{\theta}(\mathbf{x},0)d\mathbf{x}, 
\end{align*}
which is equivalent to adding the following loss term,
\begin{align*}
	-\int_{\Omega} \bar{\rho}^0(\mathbf{x}) \log \bar{\rho}_{\theta}(\mathbf{x},0)d\mathbf{x} \approx -\frac{1}{J}\sum_{j=1}^{J} 
	 \log \bar{\rho}_{\theta}(\mathbf{x}^j,0), \, \mathbf{x}^j \sim \bar{\rho}^0(\mathbf{x}).
\end{align*}
 
\subsection{Time-dependent KRnet}
KRnet \cite{tang2020deep} is a flow-based generative model, which has a stronger expressive ability than RealNVP \cite{dinh2016density} and its training is more stable. We extend the KRnet to a time-dependent KRnet to construct the invertible network described in \eqref{iNN}. Subsequently, the PDF $\bar{\rho}$ in \eqref{particle_NN} is approximated as $\bar{\rho}_{\theta}$, which is computed through \eqref{pdf_compute}. 
Here, $\theta$ represents the trainable parameters of the time-dependent KRnet.

$\Phi^{-1}_{\theta}$ in \eqref{iNN} is constructed by stacking a sequence of time-dependent bijections. The key ingredient of the bijections is time-dependent affine coupling layers.  Let the input of the affine coupling layer be $\mathbf{\bar{x}}=[\mathbf{x}_{1},\mathbf{x}_{2}]^\mathsf{T}\in \mathbb{R}^m, \mathbf{x}_{1} \in \mathbb{R}^{m_1}, \mathbf{x}_{1} \in \mathbb{R}^{m-m_1}, m\leq d$.
The output of the affine coupling layer $\tilde{\mathbf{x}}=[\tilde{\mathbf{x}}_{1},\tilde{\mathbf{x}}_{2}]^\mathsf{T}$ is defined by
\begin{equation} \label{eqn_new_affine_coupling}
	\begin{aligned}
		\tilde{\mathbf{x}}_{1} &= \mathbf{x}_{1}, \\
		\tilde{\mathbf{x}}_{2} &= \mathbf{x}_{2} \odot \left(1 + \alpha \ \mathrm{tanh}(\mathbf{s}(\mathbf{x}_{1},t)) \right) + e^{\odot\mathbf{\beta}_1} \odot \mathrm{tanh}(\mathbf{t}(\mathbf{x}_{1},t)),
	\end{aligned}
\end{equation}
where $0<\alpha<1$ is a hyperparameter set to 0.6 in our numerical experiments, the parameter $\mathbf{\beta}_1 \in \mathbb{R}^{m-m_1}$ is trainable, and $\odot$ denotes the element-wise product.
 $(\mathbf{s},\mathbf{t})$ is usually modeled by a fully connected neural network $\mathsf{NN}_1$,
\begin{equation} \label{NN1}
	(\mathbf{s}, \mathbf{t}) = \mathsf{NN}_1(\mathbf{x}_{1},t).
\end{equation}
Since \eqref{eqn_new_affine_coupling} only updates ${\mathbf{x}}_{2}$, the following affine coupling layer
is required to update ${\mathbf{x}}_{1}$,
\begin{equation}
	\begin{aligned}
		\tilde{\tilde{\mathbf{x}}}_1 &= \tilde{\mathbf{x}}_{1} \odot \left( 1 + \alpha \ \mathrm{tanh}(\tilde{\mathbf{s}}(\tilde{\mathbf{x}}_{2},t)) \right) + e^{\odot\mathbf{\beta}_2} \odot \mathrm{tanh} \left( \tilde{\mathbf{t}}(\tilde{\mathbf{x}}_{2},t) \right), \\
		\tilde{\tilde{\mathbf{x}}}_{2} &= \tilde{\mathbf{x}}_{2},
	\end{aligned}
\end{equation}
where the parameter $\mathbf{\beta}_2 \in \mathbb{R}^{m_1}$ is trainable, and $(\tilde{\mathbf{s}},\tilde{\mathbf{t}})$ is the output of a fully connected neural network $\mathsf{NN}_2$,
\begin{equation} \label{NN2}
(\tilde{\mathbf{s}},\tilde{\mathbf{t}})= \mathsf{NN}_2(\tilde{\mathbf{x}}_{2},t).
\end{equation}
More details about the time-dependent KRnet can be found in \cite{tang2020deep,he2025adaptive}.

\subsection{Time marching strategy}
Neural networks often face challenges in approximating solutions of time-dependent PDEs, particularly when the time domain is prolonged or when the solutions change rapidly over time \cite{meng2020ppinn,penwarden2023unified}. To mitigate the challenges, a time marching strategy is applied to improve the accuracy of approximate solutions \cite{zhao2021solving,he2025adaptive}.
Its core idea is to divide the entire time domain into smaller intervals and then train a separate neural network for each of these intervals. 

Let $[0,T]$ be divided equally into $N$ time intervals, and let the time interval length be $\Delta t=\frac{T}{N}$. For each time interval $((n-1)\Delta t, n \Delta t],n=1,\dots,N$, a current time-dependent KRnet with parameters $\theta_n$ is trained, and the current PDF $\bar{\rho}_{\theta_n}(\mathbf{x},t)$ is computed by
\begin{align}
    \bar{\rho}_{\theta_n}(\mathbf{x},t)=p_z(\Phi_{\theta_n}^{-1}(\mathbf{x},t)) \det\vert \nabla_x \Phi_{\theta_n}^{-1}(\mathbf{x},t)\vert,\quad t\in ((n-1)\Delta t, n \Delta t].
\end{align}
Moreover, we regard the prediction from the time-dependent KRnet of the previous time interval as the initial condition for the current time-dependent KRnet. Hence, $\theta_n$ is obtained by minimizing the following loss function,
\begin{align}
	\hat{\mathcal{L}}(\theta_n)=&\frac{\Delta t}{I}\frac{1}{J}\sum_{i=1}^I\sum_{j=1}^J \Big\Vert \frac{d\mathbf{x}^{i,j}}{dt}-\Big(\mathbf{v}-\mu\nabla \log\bar{\rho}_{\theta_n}(\mathbf{x}^{i,j},t_i)-\chi M \frac{1}{K}\sum_{k=1}^K \nabla \mathcal{K}_{\delta}(\mathbf{x}^{i,j},\mathbf{y}^{i,k}) \Big) \Big\Vert^2 \\
	&-\frac{\beta}{J}\sum_{j'=1}^{J} 
	 \log \bar{\rho}_{\theta_n}(\mathbf{x}^{j'},(n-1)\Delta t), \nonumber
\end{align}
where $\mathbf{x}^{i,j}=\Phi_{\theta_n}(\mathbf{z}^{i,j},t_i)$, $\mathbf{y}^{i,k}=\Phi_{\theta_n}(\mathbf{z}^{i,k},t_i)$, $t_i\sim U[(n-1)\Delta t,n \Delta t]$, $\mathbf{z}^{i,j}\sim p_{\mathbf{z}}(\mathbf{z})$, $\mathbf{z}^{i,k}\sim p_{\mathbf{z}}(\mathbf{z})$, $i=1,\dots,I$, $j=1,\dots,J$, $k=1,\dots,K$
, $\mathbf{x}^{j'} \sim \bar{\rho}_{\theta^*_{n-1}}(\mathbf{x},(n-1)\Delta t)$, $j'=1,\dots,J$, and $\theta^*_{n-1}=\underset{{\theta_{n-1}}}{\arg \min} \ \hat{\mathcal{L}}(\theta_{n-1})$. $\beta$ is a penalty parameter for initial conditions. 

The details of solving the KS equation \eqref{fks} are summarized in \Cref{alg}. The output of this algorithm is particles $\{\mathbf{x}_n^j\}_{j=1}^{J_s}$ at time $n\Delta t$ and the approximate solution $\rho_{\theta_n^*}(\mathbf{x},t)=M\bar{\rho}_{\theta_n^*}(\mathbf{x},t), t\in\Big((n-1)\Delta t,n\Delta t\Big]$, where $\theta_n^*$ represents the optimal parameters of the time-dependent KRnet, for $n=1,2,\dots,N$.
\begin{algorithm}[!htb]
	\caption{DeepLagrangian}
	\label{alg}
	\begin{algorithmic}[1]
		\REQUIRE The final time $T$, time interval length $\Delta t$, the number of time intervals $N=T/\Delta t$, the regularization parameter $\delta$,
		the total mass $M$, parameters of KS system  $(\mu,\chi)$, advection term $\mathbf{v}$, initial condition $\rho^0(\mathbf{x})$, 
		the number of time points for each time interval $I$, the number of spatial points  $J$, sample size for computing the chemo-attract term $K$, the prior of the time-dependent KRnet $p_z(z)$, batch size $n_{batch}$, maximum epoch number $E$, penalty parameter $\beta$, learning rate $\eta$, the number of particles $J_s$.
         \STATE Set $\bar{\rho}_{\theta_{0}^*}(\mathbf{x},0)=\bar{\rho}^0(\mathbf{x})=\rho^0(\mathbf{x})/M$.
		 \FOR {$ n= 1:N$}
		 \STATE Generate $\mathcal{T}:=\{t_i\}_{i=1}^{I}$, $t_i \sim U[(n-1)\Delta t,n\Delta t]$ and  for each $t_i$, generate $\mathcal{Z}_i:=\{\mathbf{z}^{i,j}\}_{j=1}^{J}$, $\tilde{\mathcal{Z}}_i:=\{\mathbf{z}^{i,k}\}_{k=1}^{K}$ , $\mathbf{z}^{i,j}\sim p_{\mathbf{z}}(\mathbf{z})$, $\mathbf{z}^{i,k}\sim p_{\mathbf{z}}(\mathbf{z})$.
		 \STATE Obtain the training dataset $\mathcal{D}=\{t_i,\mathcal{Z}_i\}_{i=1}^{I}$, $\tilde{\mathcal{D}}=\{t_i,\tilde{\mathcal{Z}}_i\}_{i=1}^{I}$, and generate $\mathcal{X}_0=\{\mathbf{x}^j\}_{j'=1}^J, \mathbf{x}^{j'} \sim \bar{\rho}_{\theta_{n-1}^*}(\mathbf{x},(n-1)\Delta t)$.
		 \STATE  Divide $\mathcal{D}$, $\tilde{\mathcal{D}}$ into $N_b$ mini-batches  $\{\mathcal{D}_{n_b}\}_{n_b=1}^{N_b}$, $\{\tilde{\mathcal{D}}_{n_b}\}_{n_b=1}^{N_b}$ respectively, where $N_b=\frac{I}{n_{batch}}$.
		 \STATE Initialize parameters $\theta_n$ of a time-dependent KRnet.
		 \FOR {$n_e = 1:E$}
		 \FOR {$n_b=1:N_b$}
		 \STATE Compute $\mathcal{X}_{n_b}=\Phi_{\theta_n}(\mathcal{D}_{n_b})$, $\tilde{\mathcal{X}}_{n_b}=\Phi_{\theta_n}(\tilde{\mathcal{D}}_{n_b})$ .
		 \STATE Compute the loss:
		 \begin{align*}
			&\hat{\mathcal{L}}_{batch}(\theta_n)\\
            =&\frac{\Delta t}{n_{batch}}\frac{1}{J}\sum_{i=1}^{n_{batch}}\sum_{j=1}^J \Big\Vert \frac{d\mathbf{x}^{i,j}}{dt}-\Big(\mathbf{v}-\mu\nabla \log\bar{\rho}_{\theta_n}(\mathbf{x}^{i,j},t_i)-\chi M \frac{1}{K}\sum_{k=1}^K \nabla \mathcal{K}_{\delta}(\mathbf{x}^{i,j},\mathbf{y}^{i,k}) \Big) \Big\Vert^2\\
			&-\frac{\beta}{J}\sum_{j'=1}^{J} 
			 \log \bar{\rho}_{\theta_n}(\mathbf{x}^{j'},(n-1)\Delta t), \ \mathbf{x}^{i,j}\in \mathcal{X}_{n_b}, \mathbf{y}^{i,k}\in \tilde{\mathcal{X}}_{n_b}, \mathbf{x}^{j'} \sim \bar{\rho}_{\theta^*_{n-1}}(\mathbf{x},(n-1)\Delta t), \nonumber
		\end{align*}
		and its gradient $\nabla_{\theta_n}\hat{\mathcal{L}}_{batch}(\theta_n)$.
		 \STATE Update the parameters $\theta_n$ using gradient-based optimization algorithms (e.g., Adam optimizer \cite{kingma2014adam} with learning rate $\eta$).
		 \ENDFOR
		 \ENDFOR
		 \STATE $\theta_n^*=\theta_n$ where $\theta_n$ includes the parameters of the time-dependent KRnet at the last epoch.
         \STATE Generate particles $\mathbf{x}_n^j=\Phi_{\theta_n^*}(z^j,n \Delta t ),\mathbf{z}^j\sim p_{\mathbf{z}}(\mathbf{z}),j=1,\dots,J_s$.
		 \ENDFOR
		\ENSURE Particles $\{\mathbf{x}_n^j\}_{j=1}^{J_s}$ at time $n\Delta t$
        and the approximate solution $\rho_{\theta_n^*}(\mathbf{x},t)=M\bar{\rho}_{\theta_n^*}(\mathbf{x},t), t\in\Big((n-1)\Delta t,n\Delta t\Big], n=1,2,\dots,N$.
	\end{algorithmic}
\end{algorithm}
\subsection{Compare with Eulerian methods}\label{Eulerian}
Our method is built on the Lagrangian framework of KS equations shown in \eqref{particle_NN}, and then we define a physics-informed Lagrangian loss \eqref{La} to obtain approximate KS solutions. Two Eulerian methods are given below.

PINNs \cite{raissi2019physics} minimize the following loss over a finite domain $\Omega=[-L,L]^d,\,L>0$,
\begin{align*}
    \mathcal{L}_{\text{PINN}}(\theta)=&\int_{0}^T\int_{\Omega} \Vert  \frac{\partial \bar{\rho}_{\theta}(\mathbf{x},t)}{\partial t}-\nabla \cdot(\bar{\rho}_{\theta} \mathbf{v})-\nabla\cdot(\mu \nabla \bar{\rho}_{\theta}+\chi M\bar{\rho}_{\theta} \nabla (\mathcal{K}*\bar{\rho}_{\theta})) \Vert^2  d\mathbf{x} dt\\
    &+\beta D_{KL}\left(\bar{\rho}^0(\mathbf{x})||\bar{\rho}_{\theta}(\mathbf{x},0)\right),
\end{align*}
which directly incorporates the PDE constraint \eqref{fks} and the initial PDF, and $\beta>0$.
However, it is challenging to determine an appropriate value for $L$ such that the finite domain can adequately capture the patterns of $\bar{\rho}$.

Another method, called Adaptive-PINNs, defines a loss function 
at particles $\mathbf{x}$ sampling from current probability density function $\bar{\rho}(\mathbf{x},t)$,
\begin{align*}
   &\mathcal{L}_{\text{Adaptive}}(\theta)\\
   =&\int_{0}^T\int_{\Omega} \Vert  \frac{\partial \bar{\rho}_{\theta}(\mathbf{x},t)}{\partial t}-\nabla \cdot(\bar{\rho}_{\theta} \mathbf{v})-\nabla\cdot(\mu \nabla \bar{\rho}_{\theta}+\chi M\bar{\rho}_{\theta} \nabla (\mathcal{K}*\bar{\rho}_{\theta})) \Vert^2 \bar{\rho}_{\theta}(\mathbf{x},t) d\mathbf{x} dt\\
   &+\beta D_{KL}\left(\bar{\rho}^0(\mathbf{x})||\bar{\rho}_{\theta}(\mathbf{x},0)\right),
\end{align*}
where $\beta>0$. Computing $\mathcal{L}_{\text{Adaptive}}(\theta)$ involves second-order derivatives that have high computational complexity and may introduce numerical instability, but calculating the physics-informed Lagrangian loss $\mathcal{L}(\theta)$ in \eqref{La} requires only first-order derivatives.

\section{Theoretical analysis}
In this section, we will prove that the KL divergence $\mathbf{KL}(\bar{\rho}_{\theta},\bar{\rho})$ can be controlled by physics-informed Lagrangian loss $\mathcal{L}(\theta)$.  Following \cite{shen2023entropy}, the modulated energy is defined as 
\begin{align*}
	F(\bar{\rho}_{\theta},\bar{\rho})=\frac{1}{2} \chi M\int_{\Omega^2} \mathcal{K}(\mathbf{x}-\mathbf{y})d(\bar{\rho}_{\theta}-\bar{\rho})(\mathbf{x})d(\bar{\rho}_{\theta}-\bar{\rho})(\mathbf{y}).
\end{align*}
Noting that for any two PDFs $\bar{\rho}_{\theta}$ and $\bar{\rho}$, $F(\bar{\rho}_{\theta},\bar{\rho})\geq 0$.
Combining the KL divergence and the above modulated energy defines modulated free energy as follows:
\begin{align*}
	E(\bar{\rho}_{\theta},\bar{\rho})=\mu \mathbf{KL}(\bar{\rho}_{\theta},\bar{\rho})+F(\bar{\rho}_{\theta},\bar{\rho}).
\end{align*}
We first prove the following three lemmas.
\begin{lemma}\label{lem1}
	Assume that $\bar{\rho}_{\theta}$ and $\bar{\rho}$ are the solutions to \eqref{fks_density} and \eqref{fks_approx}, respectively. The following equation holds,
	\begin{align}
\frac{d}{dt}\mathbf{KL}(\bar{\rho}_{\theta},\bar{\rho})=&-\mu \int_{\Omega}\bar{\rho}_{\theta}\Vert \nabla \log \frac{\bar{\rho}_{\theta}}{\bar{\rho}}\Vert^2-\chi M \int_{\Omega}\bar{\rho}_{\theta}\nabla \mathcal{K} * (\bar{\rho}_{\theta}-\bar{\rho})\cdot \nabla \log \frac{\bar{\rho}_{\theta}}{\bar{\rho}} \nonumber\\
&+\int_{\Omega} \bar{\rho}_{\theta} \epsilon(\mathbf{x},t)\cdot \nabla \log \frac{\bar{\rho}_{\theta}}{\bar{\rho}}. \label{lem1_eq}
	\end{align}
\end{lemma}
\begin{proof}
Based on \eqref{fks_density}, \eqref{fksA_density} and $\int_{\Omega} \frac{\partial \bar{\rho}_{\theta}}{\partial t}=0$, we have
	\begin{align*}
		&\frac{d}{dt}\mathbf{KL}(\bar{\rho}_{\theta},\bar{\rho})=\frac{d}{dt}\int_{\Omega}\bar{\rho}_{\theta} \log \frac{\bar{\rho}_{\theta}}{\bar{\rho}}\\
		=&\int_{\Omega} \frac{\partial \bar{\rho}_{\theta}}{\partial t} \log \frac{\bar{\rho}_{\theta}}{\bar{\rho}}
		+\int_{\Omega} \frac{\partial \bar{\rho}_{\theta}}{\partial t}
		-\int_{\Omega} \frac{\bar{\rho}_{\theta}}{\bar{\rho}} \frac{\partial \bar{\rho}}{\partial t} \\
		=&\int_{\Omega} \frac{\partial \bar{\rho}_{\theta}}{\partial t} \log \frac{\bar{\rho}_{\theta}}{\bar{\rho}}
		-\int_{\Omega} \frac{\bar{\rho}_{\theta}}{\bar{\rho}} \frac{\partial \bar{\rho}}{\partial t} \\
	   =&-\int_{\Omega} \nabla \cdot \Big( \bar{\rho}_{\theta}(\mathbf{v}-\mu \nabla \log \bar{\rho}_{\theta}-\chi M \nabla \mathcal{K}*\bar{\rho}_{\theta}+\epsilon(\mathbf{x},t))\Big)\log \frac{\bar{\rho}_{\theta}}{\bar{\rho}} \\
	   &+\int_{\Omega} \frac{\bar{\rho}_{\theta}}{\bar{\rho}} \nabla \cdot \Big( \bar{\rho}(\mathbf{v}-\mu \nabla \log \bar{\rho}-\chi M \nabla \mathcal{K}*\bar{\rho})\Big) \\
       =&\underbrace{-\int_{\Omega} \nabla \cdot (\bar{\rho}_{\theta}\mathbf{v})\log \frac{\bar{\rho}_{\theta}}{\bar{\rho}}
	+\int_{\Omega} \frac{\bar{\rho}_{\theta}}{\bar{\rho}} \nabla \cdot (\bar{\rho}\mathbf{v})}_{I_1}
    +\underbrace{\int_{\Omega} \log \frac{\bar{\rho}_{\theta}}{\bar{\rho}} \nabla \cdot (\bar{\rho}_{\theta} \mu \nabla \log \bar{\rho}_{\theta})
-\int_{\Omega}  \frac{\bar{\rho}_{\theta}}{\bar{\rho}} \nabla \cdot (\bar{\rho} \mu \nabla \log \bar{\rho})}_{I_2}\\
&+\underbrace{\chi M \int_{\Omega} \nabla \cdot (\bar{\rho}_{\theta} \nabla \mathcal{K}*\bar{\rho}_{\theta} )\log  \frac{\bar{\rho}_{\theta}}{\bar{\rho}}
-\chi M\int_{\Omega}  \frac{\bar{\rho}_{\theta}}{\bar{\rho}}\nabla \cdot (\bar{\rho}\nabla \mathcal{K}*\bar{\rho})}_{I_3} \underbrace{-\int_{\Omega}  \nabla \cdot (\bar{\rho}_{\theta} \epsilon(\mathbf{x},t)) \log \frac{\bar{\rho}_{\theta}}{\bar{\rho}}}_{I_4},
	\end{align*}
where $I_1, I_2, I_3$ and $I_4$ are computed as follows.

$I_1$ is first computed through the divergence theorem
\begin{align*}
	I_1=&-\int_{\Omega} \nabla \cdot (\bar{\rho}_{\theta}\mathbf{v})\log \frac{\bar{\rho}_{\theta}}{\bar{\rho}}
	+\int_{\Omega} \frac{\bar{\rho}_{\theta}}{\bar{\rho}} \nabla \cdot (\bar{\rho}\mathbf{v})\\
=&\int_{\Omega} \bar{\rho}_{\theta}\mathbf{v} \cdot \nabla \log \frac{\bar{\rho}_{\theta}}{\bar{\rho}}
-\int_{\Omega} \nabla \frac{\bar{\rho}_{\theta}}{\bar{\rho}} \cdot \bar{\rho}\mathbf{v} \\
=&0.
\end{align*}
$I_2$ is also obtained through the divergence theorem
\begin{align*}
	I_2=&\int_{\Omega} \log \frac{\bar{\rho}_{\theta}}{\bar{\rho}} \nabla \cdot (\bar{\rho}_{\theta} \mu \nabla \log \bar{\rho}_{\theta})
-\int_{\Omega}  \frac{\bar{\rho}_{\theta}}{\bar{\rho}} \nabla \cdot (\bar{\rho} \mu \nabla \log \bar{\rho})\\
=&-\mu \int_{\Omega} \bar{\rho}_{\theta} \nabla \log \bar{\rho}_{\theta} \cdot \nabla \log \frac{\bar{\rho}_{\theta}}{\bar{\rho}}
+\mu \int_{\Omega}  \bar{\rho} \nabla \log \bar{\rho}\cdot \nabla  \frac{\bar{\rho}_{\theta}}{\bar{\rho}}\\
=&-\mu \int_{\Omega} \bar{\rho}_{\theta} \Vert \nabla  \log  \frac{\bar{\rho}_{\theta}}{\bar{\rho}} \Vert^2.
\end{align*}
In the following, $I_3$ is defined and calculated by applying the divergence theorem
\begin{align*}
	I_3=&\chi M \int_{\Omega} \nabla \cdot (\bar{\rho}_{\theta} \nabla \mathcal{K}*\bar{\rho}_{\theta} )\log  \frac{\bar{\rho}_{\theta}}{\bar{\rho}}
-\chi M\int_{\Omega}  \frac{\bar{\rho}_{\theta}}{\bar{\rho}}\nabla \cdot (\bar{\rho}\nabla \mathcal{K}*\bar{\rho})\\
=&-\chi M \int_{\Omega} (\bar{\rho}_{\theta} \nabla \mathcal{K}*\bar{\rho}_{\theta} ) \cdot \nabla \log  \frac{\bar{\rho}_{\theta}}{\bar{\rho}}
+\chi M\int_{\Omega} \nabla \frac{\bar{\rho}_{\theta}}{\bar{\rho}} \cdot (\bar{\rho}\nabla \mathcal{K}*\bar{\rho})\\
=& -\chi M \int_{\Omega} \bar{\rho}_{\theta} \nabla \mathcal{K}*(\bar{\rho}_{\theta}-\bar{\rho}) \cdot \nabla \log  \frac{\bar{\rho}_{\theta}}{\bar{\rho}}.
\end{align*}
Finally, we compute $I_4$
\begin{align*}
	I_4=&-\int_{\Omega}  \nabla \cdot (\bar{\rho}_{\theta} \epsilon(\mathbf{x},t)) \log \frac{\bar{\rho}_{\theta}}{\bar{\rho}}\\
	=& \int_{\Omega} (\bar{\rho}_{\theta} \epsilon(\mathbf{x},t)) \cdot \nabla \log \frac{\bar{\rho}_{\theta}}{\bar{\rho}}.
\end{align*}
Therefore, adding $I_1, I_2, I_3, I_4$ together derives \eqref{lem1_eq}. 
\end{proof}

\begin{lemma}
	Given the same assumptions as in \Cref{lem1}, the following equation holds:
	\begin{align*}
		\frac{d}{dt}F(\bar{\rho}_{\theta},\bar{\rho})=&-\chi^2 M^2\int_{\Omega} \bar{\rho}_{\theta} \Vert \nabla \mathcal{K}*(\bar{\rho}_{\theta}-\bar{\rho})\Vert^2
	+\chi M\int_{\Omega}\bar{\rho}_{\theta}\epsilon(\mathbf{x},t)\cdot \nabla\mathcal{K}*(\bar{\rho}_{\theta}-\bar{\rho})\\
	&-\mu\chi M \int_{\Omega}\bar{\rho}_{\theta}\nabla\mathcal{K}*(\bar{\rho}_{\theta}-\bar{\rho})\cdot \nabla \log \frac{\bar{\rho}_{\theta}}{\bar{\rho}}\\
&+\chi M\frac{1}{2}\int_{\Omega^2}\nabla \mathcal{K}(\mathbf{x}-\mathbf{y})\cdot(\mathcal{A}[\bar{\rho}](\mathbf{x})-\mathcal{A}[\bar{\rho}](\mathbf{y}))d (\bar{\rho}_{\theta}-\bar{\rho})^{\otimes 2}(\mathbf{x},\mathbf{y}).
\end{align*}
\end{lemma}
\begin{proof}
We first use that $\mathcal{K}$ is an even function and apply \eqref{fks_density}, \eqref{fksA_density} to see that
	\begin{align*}
		&\frac{d}{dt}F(\bar{\rho}_{\theta},\bar{\rho})=\frac{1}{2}\chi M\frac{d}{dt}\int_{\Omega^2} \mathcal{K}(\mathbf{x}-\mathbf{y})d (\bar{\rho}_{\theta}-\bar{\rho})^{\otimes 2}(\mathbf{x},\mathbf{y})\\
=& \chi M \int_{\Omega} \mathcal{K} *  (\bar{\rho}_{\theta}-\bar{\rho})(\mathbf{x}) (\frac{\partial\bar{\rho}_{\theta}}{\partial t}-\frac{\partial\bar{\rho}}{\partial t})(\mathbf{x}) d\mathbf{x}\\
=& \chi M \int_{\Omega} \mathcal{K} *  (\bar{\rho}_{\theta}-\bar{\rho})(\mathbf{x}) \nabla \cdot \Big( \bar{\rho}_{\theta}\left(-\mathbf{v}+\mu \nabla \log \bar{\rho}_{\theta}+\chi M \nabla \mathcal{K}*\bar{\rho}_{\theta}-\epsilon(\mathbf{x},t)\right)\\
&-\bar{\rho}\left(-\mathbf{v}+\mu\nabla \log \bar{\rho}+\chi M \nabla \mathcal{K}*\bar{\rho}\right)\Big)d\mathbf{x}\\
=&\underbrace{-\chi M   \int_{\Omega} \mathcal{K} *  (\bar{\rho}_{\theta}-\bar{\rho})(\mathbf{x})
	\nabla \cdot (\bar{\rho}_{\theta}-\bar{\rho})\mathbf{v}}_{J_1}\\
    &+\underbrace{\chi M \mu \int_{\Omega} \mathcal{K}*(\bar{\rho}_{\theta}-\bar{\rho})(\mathbf{x})\nabla \cdot (\bar{\rho}_{\theta}\nabla \log \bar{\rho}_{\theta})
-\chi M \mu \int_{\Omega} \mathcal{K}*(\bar{\rho}_{\theta}-\bar{\rho})(\mathbf{x})\nabla \cdot (\bar{\rho}\nabla \log \bar{\rho})}_{J_2}\\
&+\underbrace{\chi^2 M^2 \int_{\Omega} \mathcal{K}*(\bar{\rho}_{\theta}-\bar{\rho})\nabla \cdot(\bar{\rho}_{\theta}\nabla \mathcal{K}*\bar{\rho}_{\theta}-\bar{\rho}\nabla \mathcal{K}*\bar{\rho})}_{J_3}
\underbrace{-\chi M \int_{\Omega} \mathcal{K}*(\bar{\rho}_{\theta}-\bar{\rho})\nabla \cdot (\bar{\rho}_{\theta}\epsilon(\mathbf{x},t))}_{J_4}.
\end{align*}
Here, $J_1, J_2, J_3$ and $J_4$ are calculated as follows.
$J_1$ means that 
\begin{align*}
	J_1=&-\chi M   \int_{\Omega} \mathcal{K} *  (\bar{\rho}_{\theta}-\bar{\rho})(\mathbf{x})
	\nabla \cdot (\bar{\rho}_{\theta}-\bar{\rho})\mathbf{v}\\
=&\chi M   \int_{\Omega} \nabla \mathcal{K} *  (\bar{\rho}_{\theta}-\bar{\rho})(\mathbf{x}) 
\cdot (\bar{\rho}_{\theta}-\bar{\rho})\mathbf{v}\\
=&\chi M \frac{1}{2}  \int_{\Omega} \nabla\mathcal{K}(\mathbf{x}-\mathbf{y})\cdot (\mathbf{v}(\mathbf{x})-\mathbf{v}(\mathbf{y}))d(\bar{\rho}_{\theta}-\bar{\rho})^{\otimes 2}(\mathbf{x},\mathbf{y}),
\end{align*}
where the divergence theorem is used for deriving the second equation, and the last equation is obtained by doing the symmetrization, i.e., swapping $x$ and $y$ in the second equation to obtain another equation, using the property of the even function $\mathcal{K}$ for another equation, and then taking the average of these two equations.
$J_2$ reads
\begin{align*}
	&J_2=\chi M \mu \int_{\Omega} \mathcal{K}*(\bar{\rho}_{\theta}-\bar{\rho})(\mathbf{x})\nabla \cdot (\bar{\rho}_{\theta}\nabla \log \bar{\rho}_{\theta})
-\chi M \mu \int_{\Omega} \mathcal{K}*(\bar{\rho}_{\theta}-\bar{\rho})(\mathbf{x})\nabla \cdot (\bar{\rho}\nabla \log \bar{\rho})\\
=&-\chi M \mu \int_{\Omega} \nabla\mathcal{K}*(\bar{\rho}_{\theta}-\bar{\rho})\cdot (\bar{\rho}_{\theta}\nabla \log \bar{\rho}_{\theta})
+\chi M \mu \int_{\Omega} \nabla\mathcal{K}*(\bar{\rho}_{\theta}-\bar{\rho})\cdot (\bar{\rho}\nabla \log \bar{\rho})\\
=&-\chi M \mu \int_{\Omega} \nabla\mathcal{K}*(\bar{\rho}_{\theta}-\bar{\rho})\cdot (\bar{\rho}_{\theta}\nabla \log \frac{\bar{\rho}_{\theta}}{\bar{\rho}})
-\chi M \mu \int_{\Omega} \nabla\mathcal{K}*(\bar{\rho}_{\theta}-\bar{\rho})\cdot ( (\bar{\rho}_{\theta}-\bar{\rho}) \nabla\log \bar{\rho})\\
=&-\chi M \mu \int_{\Omega} \bar{\rho}_{\theta}\nabla\mathcal{K}*(\bar{\rho}_{\theta}-\bar{\rho})\cdot \nabla \log \frac{\bar{\rho}_{\theta}}{\bar{\rho}}\\
&-\frac{\chi M \mu }{2} \int_{\Omega} \nabla \mathcal{K}(\mathbf{x}-\mathbf{y}) \cdot (\nabla \log \bar{\rho}(\mathbf{x})-\nabla \log \bar{\rho}(\mathbf{y}))d(\bar{\rho}_{\theta}-\bar{\rho})^{\otimes 2}(\mathbf{x},\mathbf{y}),
\end{align*}
where the second equation and the last equation are derived by the divergence theorem and doing symmetrization, respectively. $J_3$ is given by 
\begin{align*}
	J_3=&\chi^2 M^2 \int_{\Omega} \mathcal{K}*(\bar{\rho}_{\theta}-\bar{\rho})\nabla \cdot(\bar{\rho}_{\theta}\nabla \mathcal{K}*\bar{\rho}_{\theta}-\bar{\rho}\nabla \mathcal{K}*\bar{\rho})\\
	=&-\chi^2 M^2 \int_{\Omega}  \nabla \mathcal{K}*(\bar{\rho}_{\theta}-\bar{\rho}) \cdot (\bar{\rho}_{\theta}\nabla \mathcal{K}*\bar{\rho}_{\theta}-\bar{\rho}\nabla \mathcal{K}*\bar{\rho})\\
	=&-\chi^2 M^2 \int_{\Omega} \nabla \mathcal{K}*(\bar{\rho}_{\theta}-\bar{\rho}) \cdot (\bar{\rho}_{\theta}\nabla \mathcal{K}*(\bar{\rho}_{\theta}-\bar{\rho}))\\
	&-\chi^2 M^2 \int_{\Omega}  \nabla \mathcal{K}*(\bar{\rho}_{\theta}-\bar{\rho}) \cdot (\bar{\rho}_{\theta}-\bar{\rho})\nabla\mathcal{K}*\bar{\rho}\\
	=&-\chi^2 M^2 \int_{\Omega}\bar{\rho}_{\theta}\Vert \nabla \mathcal{K}*(\bar{\rho}_{\theta}-\bar{\rho}) \Vert^2\\
	&-\frac{\chi^2 M^2 }{2}\int_{\Omega}\nabla \mathcal{K}(\mathbf{x}-\mathbf{y})(\nabla \mathcal{K}*\bar{\rho}(\mathbf{x})-\nabla \mathcal{K}*\bar{\rho}(\mathbf{y}))d(\bar{\rho}_{\theta}-\bar{\rho})^{\otimes 2}(\mathbf{x},\mathbf{y}).
\end{align*}
Here, the second equation and the last equation rely on the divergence theorem and the symmetrization, respectively. $J_4$ reads
\begin{align*}
	J_4=&-\chi M \int_{\Omega} \mathcal{K}*(\bar{\rho}_{\theta}-\bar{\rho})\nabla \cdot (\bar{\rho}_{\theta}\epsilon(\mathbf{x},t))\\
	=&\chi M \int_{\Omega}\bar{\rho}_{\theta}\epsilon(\mathbf{x},t)\cdot \nabla\mathcal{K}*(\bar{\rho}_{\theta}-\bar{\rho}),
\end{align*}
according to the divergence theorem. We add up $J_1, J_2, J_3$ and $J_4$ to complete the proof.
\end{proof}

\begin{lemma}\label{lem3}
	Given the same assumptions as in \Cref{lem1}, the following equation holds:
	\begin{align*}
		\frac{d}{dt}E(\bar{\rho}_{\theta},\bar{\rho})=&
		-\int_{\Omega}\bar{\rho}_{\theta}\Vert \chi M \nabla \mathcal{K}*(\bar{\rho}_{\theta}-\bar{\rho})+\mu \nabla \log \frac{\bar{\rho}_{\theta}}{\bar{\rho}}\Vert^2\\
	&+\int_{\Omega} \bar{\rho}_{\theta} \epsilon(\mathbf{x},t)\cdot  \Big(\chi M \nabla\mathcal{K}*(\bar{\rho}_{\theta}-\bar{\rho})+\mu \nabla \log \frac{\bar{\rho}_{\theta}}{\bar{\rho}}\Big)\\
	&+\chi M\frac{1}{2}\int_{\Omega^2}\nabla \mathcal{K}(\mathbf{x}-\mathbf{y})\cdot(\mathcal{A}[\bar{\rho}](\mathbf{x})-\mathcal{A}[\bar{\rho}](\mathbf{y}))d (\bar{\rho}_{\theta}-\bar{\rho})^{\otimes 2}(\mathbf{x},\mathbf{y}).
	\end{align*}
\end{lemma}

Based on the above analysis, we can prove the following main result. 
\begin{theorem}\label{therorem 4}
	Assume that  $\bar{\rho}_{\theta}$ satisfies the initial PDF $\bar{\rho}^0(\mathbf{x})$. Furthermore, suppose that for $t\in [0,T]$, the underlying velocity $\mathcal{A}[\bar{\rho}](\mathbf{x})$
	is Lipschitz in $x$ and $\sup_{t\in [0,T]}\Vert \nabla \mathcal{A}[\bar{\rho}](\cdot)\Vert_{L^{\infty}}=C_1 < \infty$. Then 
	there exists $C>0$ such that
	\begin{align*}
		\sup_{t\in [0,T]} \mu  \mathbf{KL}(\bar{\rho}_{\theta},\bar{\rho})\leq \sup_{t\in [0,T]} E(\bar{\rho}_{\theta},\bar{\rho})\leq \frac{1}{4}\exp(CC_1\chi M T)\mathcal{L}(\theta).
	\end{align*}
\end{theorem}
\begin{proof}
Applying \Cref{lem3}, we have
	\begin{align*}
		&\frac{d}{dt}E(\bar{\rho}_{\theta},\bar{\rho})=-\int_{\Omega}\bar{\rho}_{\theta}\Vert \chi M \nabla \mathcal{K}*(\bar{\rho}_{\theta}-\bar{\rho})+\mu \nabla \log \frac{\bar{\rho}_{\theta}}{\bar{\rho}}-\frac{1}{2}\epsilon(\mathbf{x},t)\Vert^2
	+\frac{1}{4}\int_{\Omega}\bar{\rho}_{\theta}\Vert \epsilon(\mathbf{x},t)\Vert^2\\
	&+\chi M\frac{1}{2}\int_{\Omega^2}\nabla \mathcal{K}(\mathbf{x}-\mathbf{y})\cdot(\mathcal{A}[\bar{\rho}](\mathbf{x})-\mathcal{A}[\bar{\rho}](\mathbf{y}))d (\bar{\rho}_{\theta}-\bar{\rho})^{\otimes 2}(\mathbf{x},\mathbf{y}).
	\end{align*}
	Lemma 5.2 \cite{bresch2019modulated} tells that if the ground truth $\mathcal{A}[\bar{\rho}]$ is Lipschitz, i.e., 
	$\mathcal{A}[\bar{\rho}]\in W^{1,\infty}$, then
	\begin{align*}
		&\chi M\frac{1}{2}\int_{\Omega^2}\nabla \mathcal{K}(\mathbf{x}-\mathbf{y})\cdot(\mathcal{A}[\bar{\rho}](\mathbf{x})-\mathcal{A}[\bar{\rho}](\mathbf{y}))d (\bar{\rho}_{\theta}-\bar{\rho})^{\otimes 2}(\mathbf{x},\mathbf{y})\\
    &\leq C \chi M \Vert\nabla \mathcal{A}[\bar{\rho}] \Vert_{L^{\infty}}F(\bar{\rho}_{\theta},\bar{\rho}).
	\end{align*}
Hence,
\begin{align*}
	\frac{d}{dt}E(\bar{\rho}_{\theta},\bar{\rho})\leq&\frac{1}{4}\int_{\Omega}\bar{\rho}_{\theta}\Vert \epsilon(\mathbf{x},t)\Vert^2
	+CC_1\chi M F(\bar{\rho}_{\theta},\bar{\rho})\\
	\leq &\frac{1}{4}\int_{\Omega}\bar{\rho}_{\theta}\Vert \epsilon(\mathbf{x},t)\Vert^2
	+CC_1\chi M E(\bar{\rho}_{\theta},\bar{\rho}).
\end{align*}
Applying Gronwall inequality, we obtain, for $\forall t \in[0,T]$
\begin{align*}
	E(\bar{\rho}_{\theta},\bar{\rho})&\leq \frac{1}{4}\int_{0}^t \exp(CC_1\chi M (t-s))\int_{\Omega}\bar{\rho}_{\theta}\Vert \epsilon(\mathbf{x},s)\Vert^2 d\mathbf{x} ds\\
&\leq \frac{1}{4}\exp(CC_1\chi M t)\int_{0}^t \int_{\Omega}\bar{\rho}_{\theta}\Vert \epsilon(\mathbf{x},s)\Vert^2 d\mathbf{x} ds\\
&\leq \frac{1}{4}\exp(CC_1\chi M T) \mathcal{L}(\theta).
\end{align*}
The proof has been completed.
\end{proof}

\section{Numerical experiments} 
In this section, we present four test problems to show the accuracy of our DeepLagrangian method. In test problem 1, we demonstrate the advantages of our proposed time-marching strategies for solving KS equations. We further conduct a comparative analysis of our Lagrangian-based method against established Eulerian methods, specifically PINNs and Adaptive-PINNs, as detailed in \Cref{Eulerian}. Test problem 2 shows the effects of different initial conditions.
In test problem 3, the advection term with different physical parameters is considered.
Test problem 4 extends the investigation to the 3D KS equation. 

To evaluate the performance of our method, we compare the histogram of the network output of our method with the reference solution, where the output is the particles sampled from time-dependent KRnet (see \Cref{alg}). The reference solution is achieved by the interacting particle methods (IPM) \cite{wang2024deepparticle}, which simulates the following SDE:
\begin{align}\label{SDE}
	d\mathbf{x}^j&=-\frac{\chi M}{J_s}\nabla_{\mathbf{x}^j} \sum_{i=1,i\neq j}^{J_s} K_{\delta}(\mathbf{x}^j,\mathbf{x}^i)dt+\mathbf{v}(\mathbf{x}^j)dt+\sqrt{2\mu}dW^j,\, j=1,2,\dots,J_s,
\end{align}
where  $\mathbf{x}^j$ is the position of the $j$th particle, $dW^j$ is independent Brownian motions and $J_s$ is the number of particles.
As $J_s$ increases to $\infty$,  the macroscopic limit (McKean–Vlasov equation) of \eqref{SDE} is \eqref{fks}. For all test problems, the number of particles is set to $J_s=10^4$. 

In the following numerical experiments, the KS system with $\chi =\mu=1$ in \eqref{fks} is considered.
In \Cref{alg}, the total mass $M$ is set to be greater than $8\pi$ such that the solution of the KS system \eqref{fks} will blow up in a finite time. The regularization parameter is chosen as $\delta=10^{-3}$. The Monte Carlo integration of the loss function employs $I=50$ time points and $J=1000$ spatial points.
The training parameters include a batch size of 5000 and a total of 2000 epochs. The penalty parameter is set to $\beta=100$. The Adam optimizer with learning rate $\eta=0.003$ is applied.

For the time-dependent KRnet, the components of $x\in\mathbb{R}^{d}$ are partitioned into $d$ equal groups,
and one group is deactivated after 8 affine coupling layers, where the bijection given by each coupling layer is based on the outputs of a fully connected neural network with two hidden layers of 48 neurons and the Swish activation function.
The neural networks used for PINNs and Adaptive-PINNs are the same as the time-dependent KRnet of our method.
All neural networks are trained on a single
NVIDIA A100 Tensor Core GPU card.
 
\subsection{Test problem 1: 2D KS system without advection}\label{test1_section}
The KS system \eqref{fks} is considered in the case of $d=2,\mathbf{v}=0$.
Our goal is to learn the change of the KS solution depending on the evolution time $t$ starting from an initial condition. The initial condition is set to an unnormalized uniform distribution on a ball with a radius of 1 centered at the origin.
The total mass is set to $M=16\pi\geq 8\pi$. The system will blow up when $t\geq 0.125$.
The final time is set to $T=0.12$. For the time-dependent KRnet, its prior $p_{\mathbf{z}}(\mathbf{z})$ is set to a standard Gaussian distribution.

First, we explain why we need to use time marching strategies. We set $I=100,\,J=1000$ in \eqref{La_em} for learning the KS solution within $[0,T]$. \Cref{problem1_long} shows the histogram of the network output of the time-dependent KRnet and the reference solution, where it can be seen that the time-dependent KRnet fails to learn the KS solution when $t>0.04$.  In order to obtain accurate KS solutions, the time-dependent KRnet is combined with time marching strategies to solve the KS equations in the rest of this paper.
\begin{figure}[!htb]
	\centering
	\subfloat[][$t=0.01$]{\includegraphics[width=.5
 \textwidth]{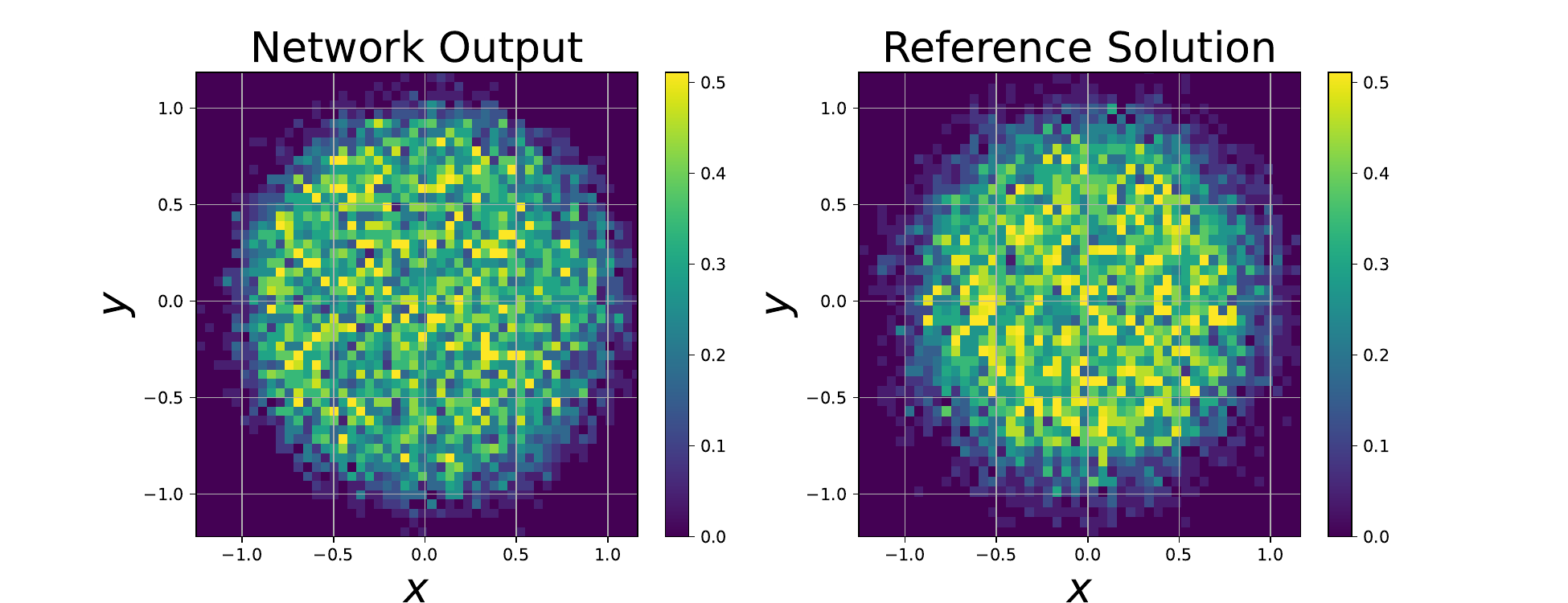}}
	\subfloat[][$t=0.03$]{\includegraphics[width=.5\textwidth]{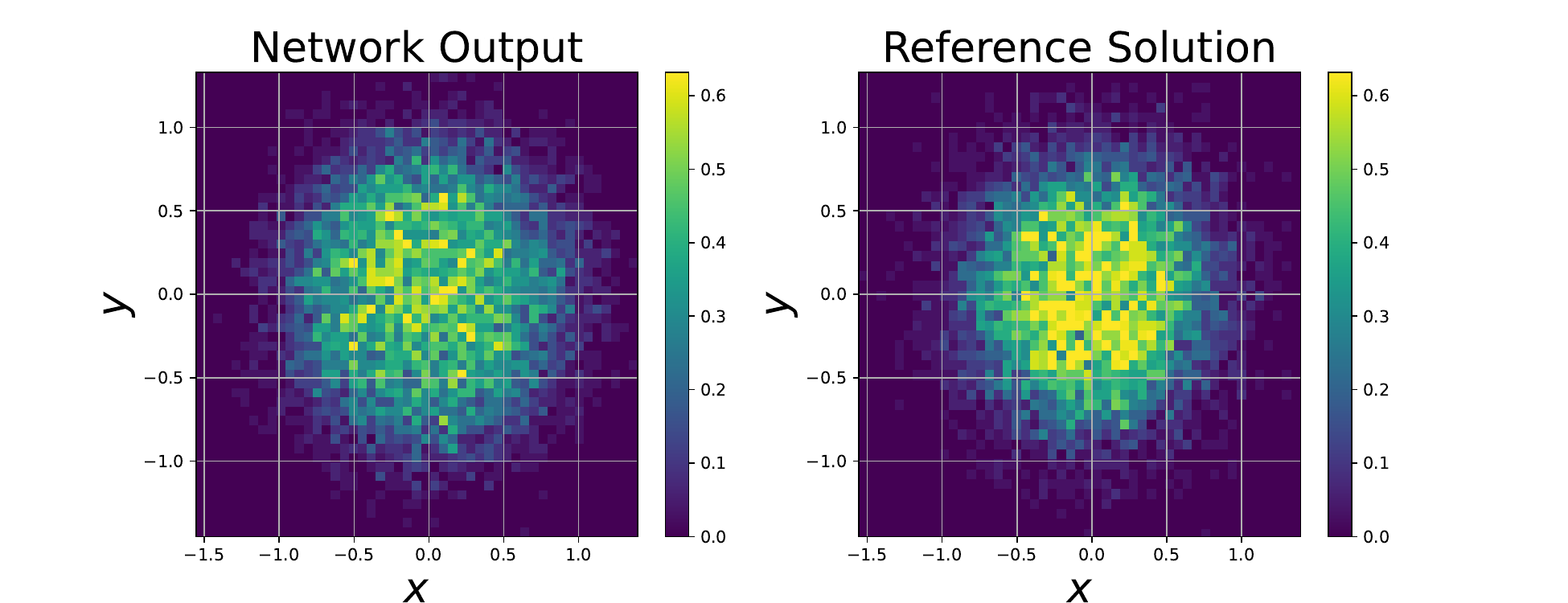}}
    \\
    \subfloat[][$t=0.04$]{\includegraphics[width=.5\textwidth]{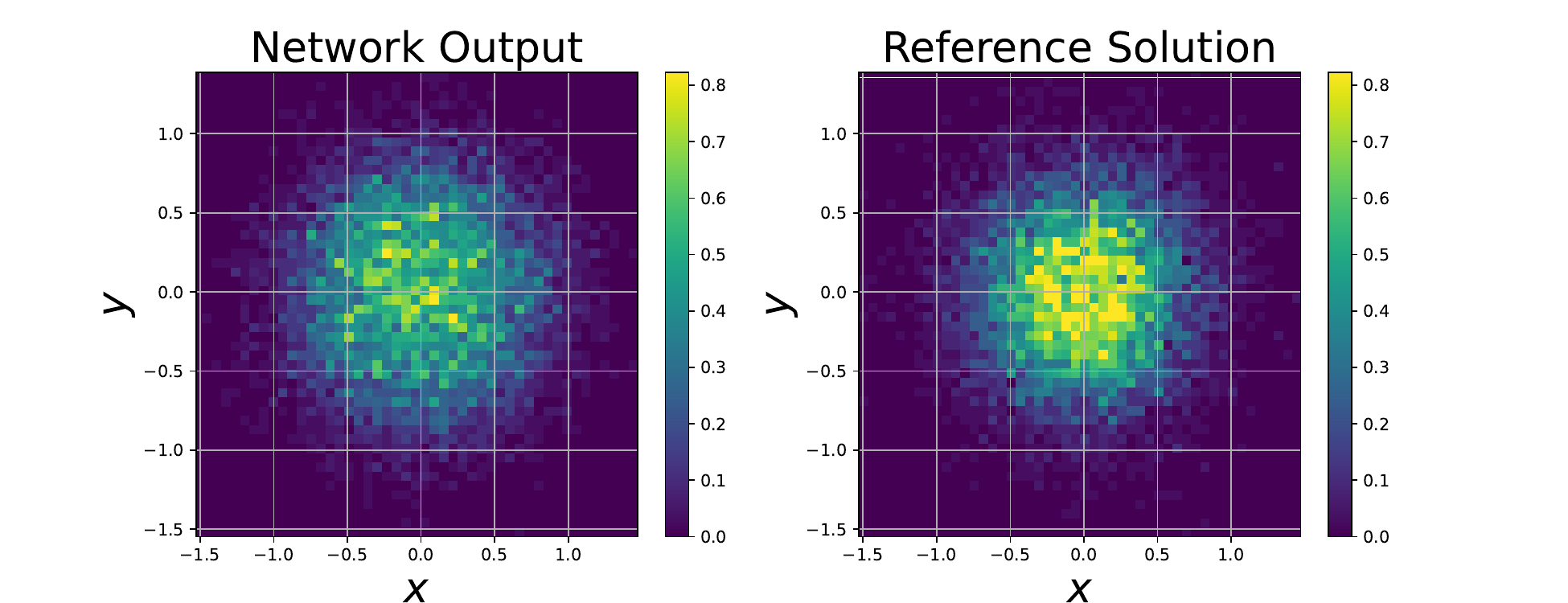}}
	\subfloat[][$t=0.05$]{\includegraphics[width=.5\textwidth]{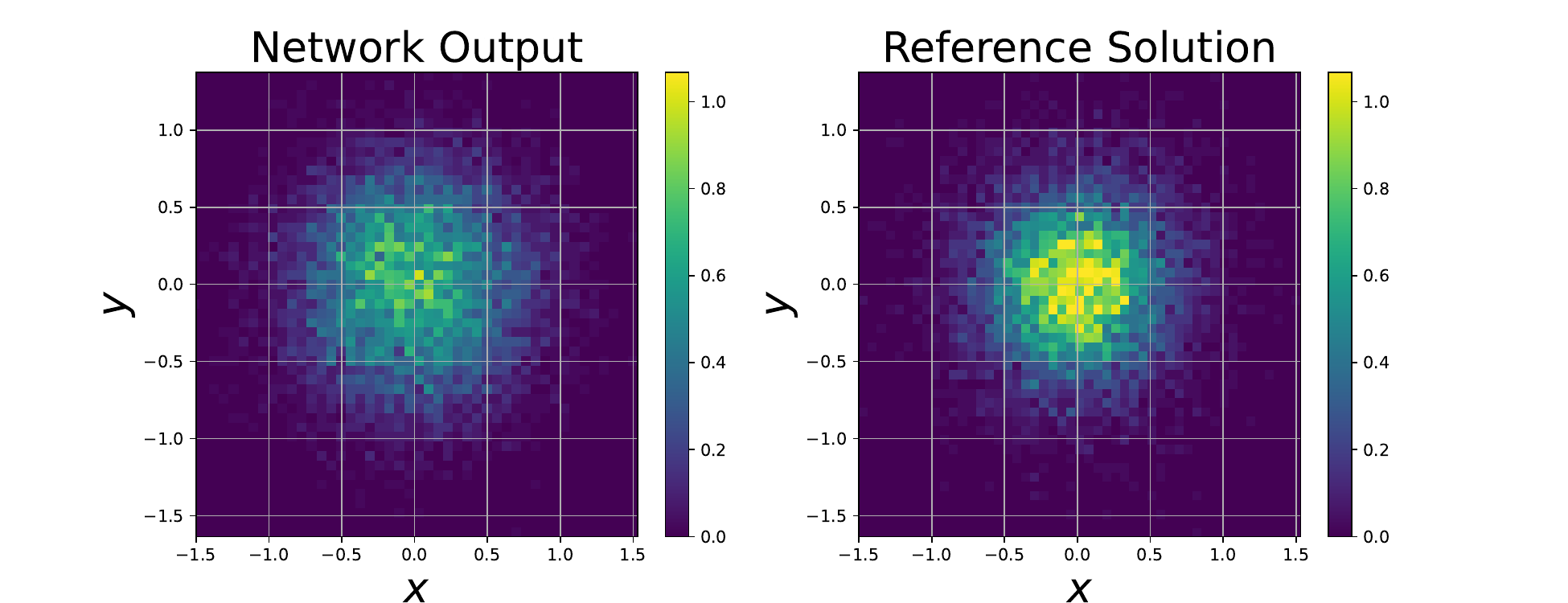}}\\
	\subfloat[][$t=0.08$]{\includegraphics[width=.5\textwidth]{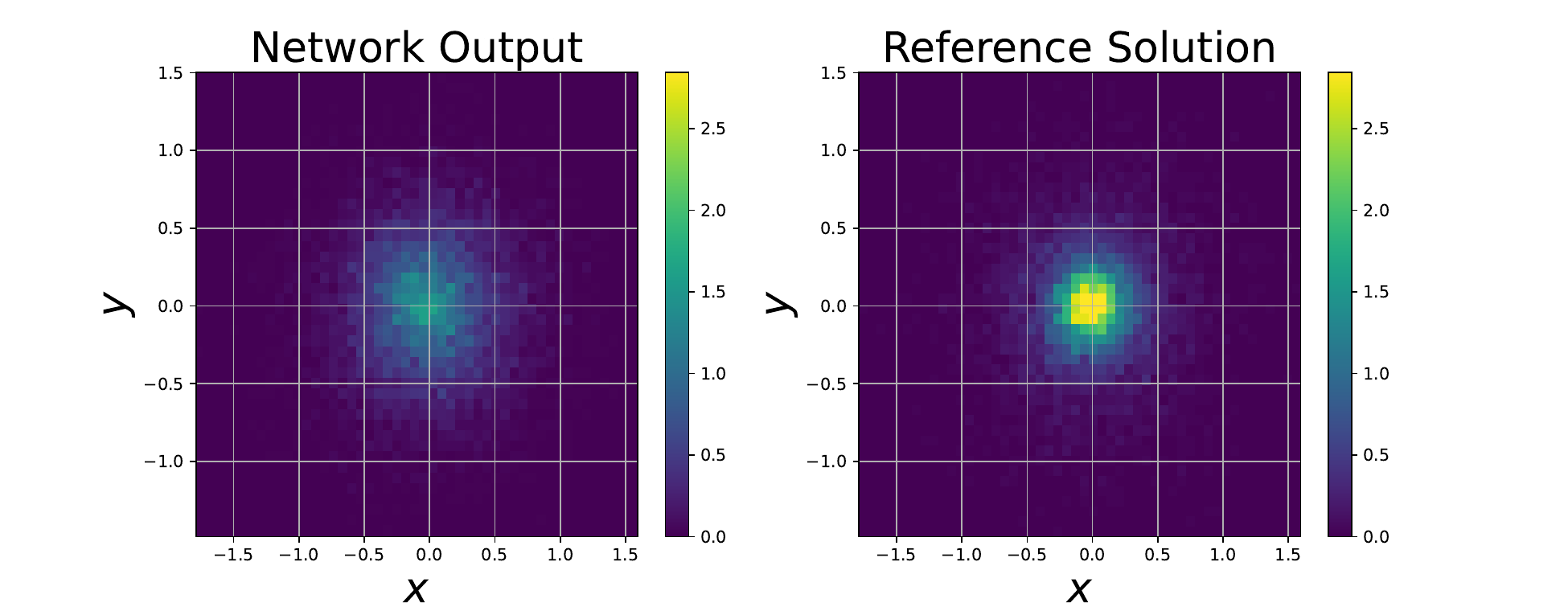}}
    \subfloat[][$t=0.1$]{\includegraphics[width=.5\textwidth]{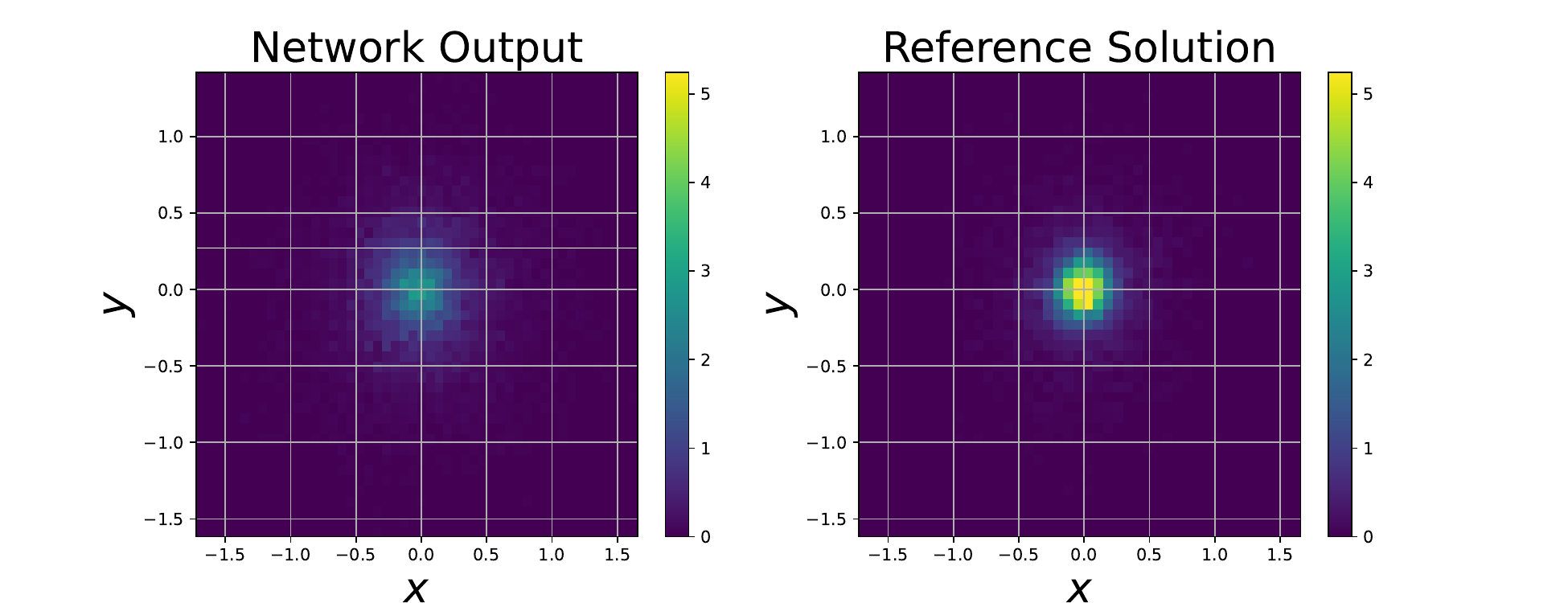}}
	\caption{2D KS solutions without time marching, test problem 1.}
    \label{problem1_long}
\end{figure}

For achieving the time marching strategies, we equally divide $[0,T]$ into four parts and then in each part, apply a time-dependent KRnet to approximate the KS solution.  \Cref{problem1_ks} shows the KS solution obtained by our DeepLagrangian method and the reference solution at different times, where it is clear that they are visually indistinguishable even near the blow-up time.
\begin{figure}[!htb]
	\centering
	\subfloat[][$t=0.01$]{\includegraphics[width=.5
 \textwidth]{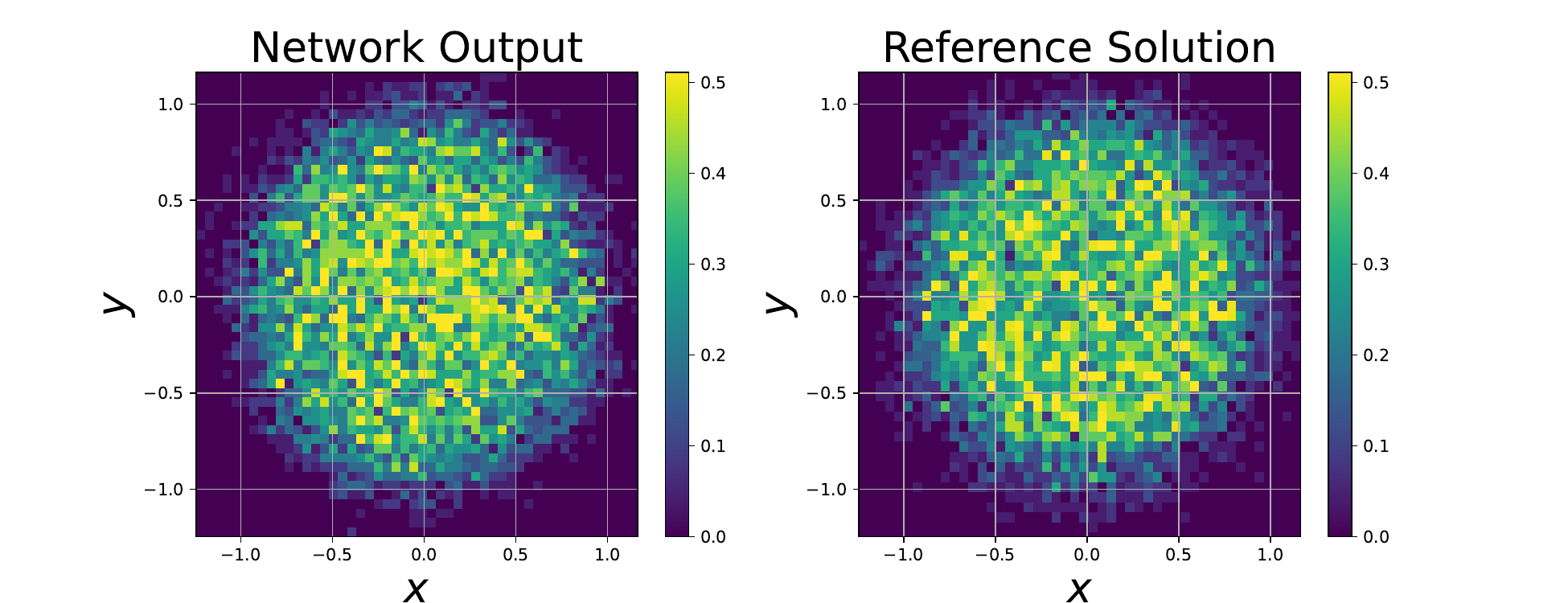}}
	\subfloat[][$t=0.04$]{\includegraphics[width=.5\textwidth]{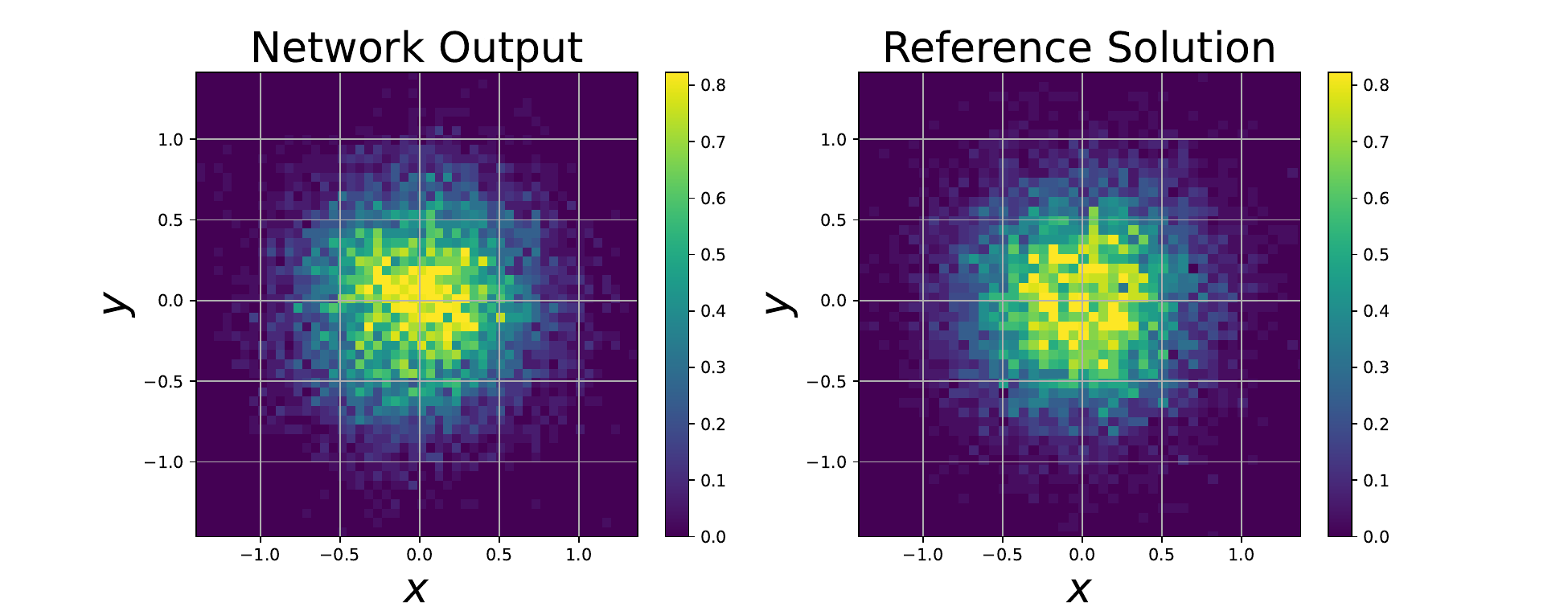}}\\
	\subfloat[][$t=0.08$]{\includegraphics[width=.5\textwidth]{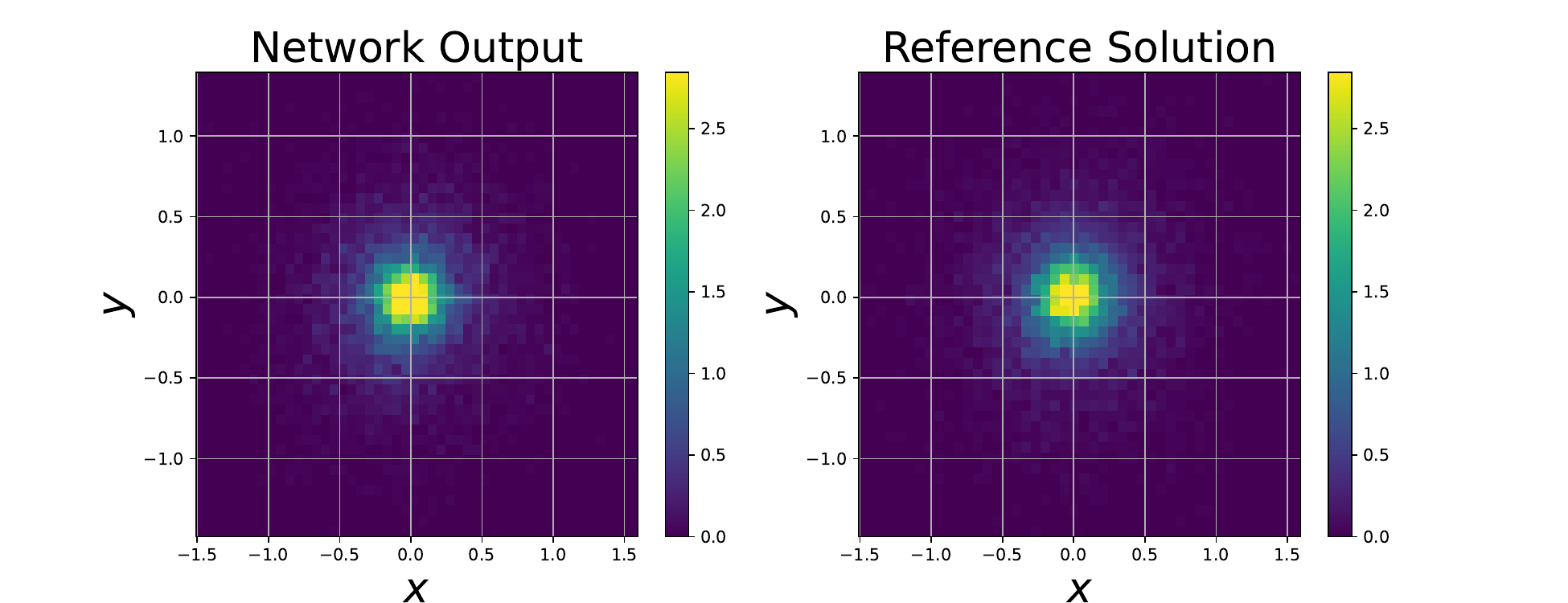}}
	\subfloat[][$t=0.12$]{\includegraphics[width=.5\textwidth]{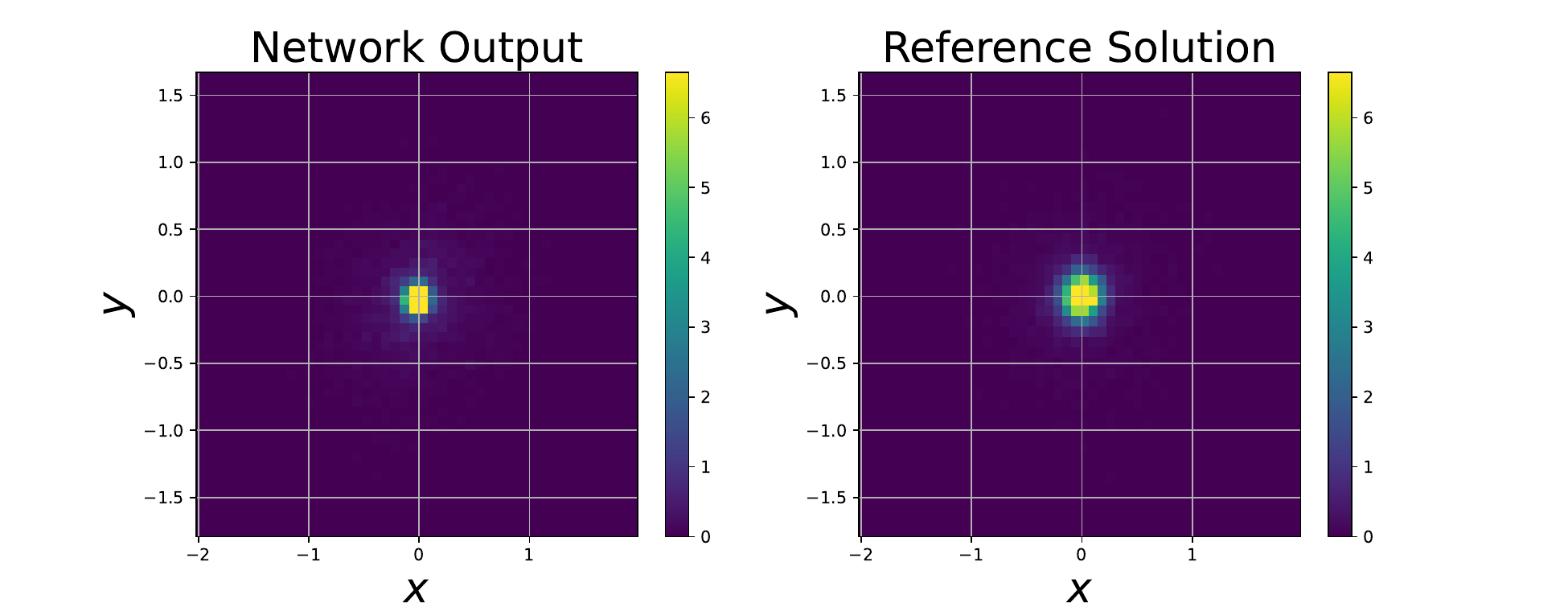}}
	\caption{2D KS solutions with time marching, test problem 1.}
    \label{problem1_ks}
\end{figure}

Then, our method is compared with PINNs and Adaptive-PINNs within the time intervals $[0,0.03]$.
\Cref{problem1_compare} illustrates the KS solutions computed by our method, PINNs, and Adaptive-PINNs, where it is obvious that our solution is close to the reference solution at $t=0.01,0.03$, but the solutions of PINNs and Adaptive-PINNs differ from the reference solution at $t=0.03$.
This is attributed to the fact that the Lagrangian framework can naturally track near-singular solutions.  
\begin{figure}[!htb]
	\centering
     \subfloat[][Our method, $t=0.01$]{\includegraphics[width=.5\textwidth]{ks1_2-eps-converted-to.pdf}}
     \subfloat[][Our method, $t=0.03$]{\includegraphics[width=.5\textwidth]{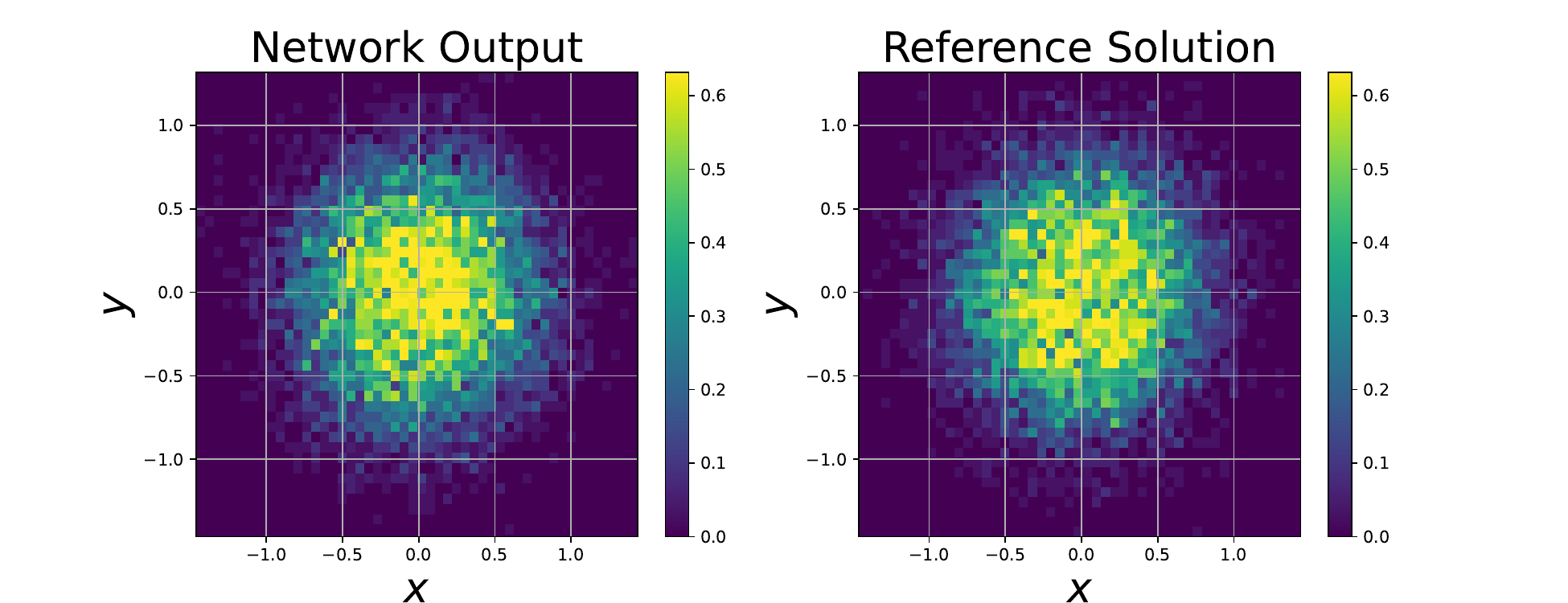}}
     \\
	\subfloat[][PINNs, $t=0.01$]{\includegraphics[width=.5
 \textwidth]{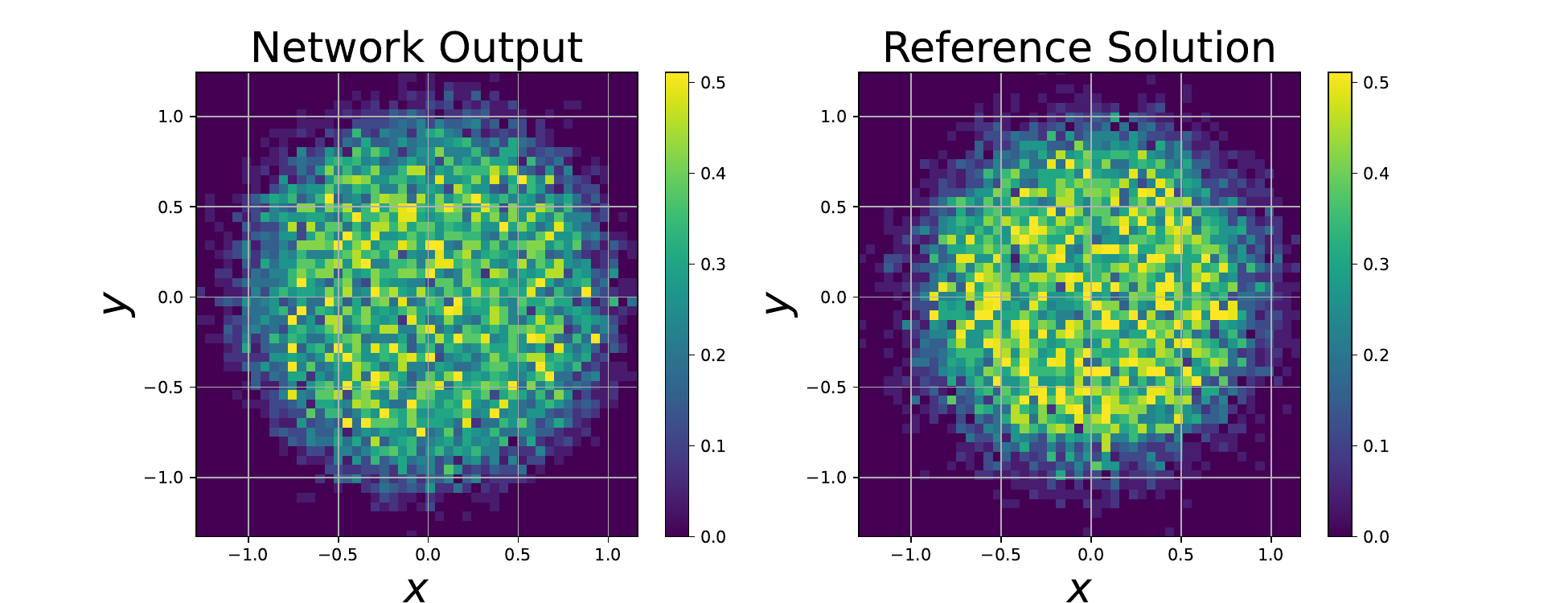}}
 \subfloat[][PINNs, $t=0.03$]{\includegraphics[width=.5\textwidth]{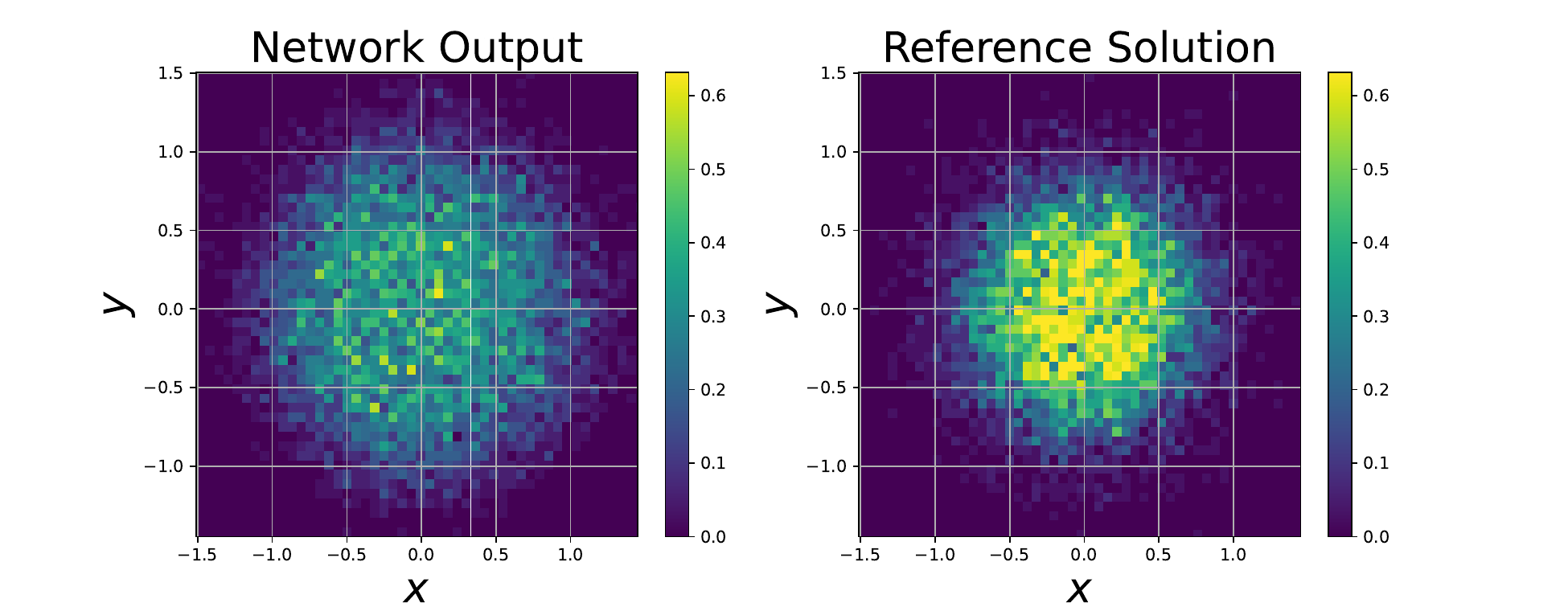}}
 \\
\subfloat[][Adaptive-PINNs , $t=0.01$]{\includegraphics[width=.5\textwidth]{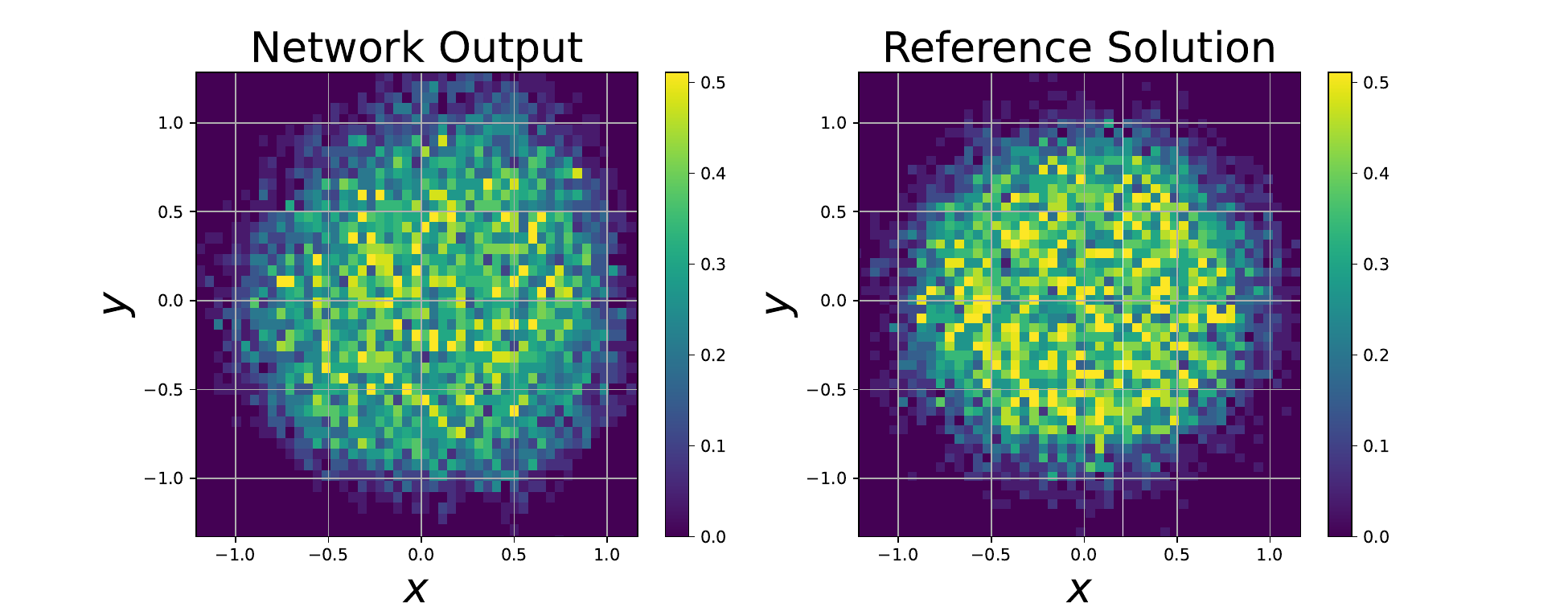}}
\subfloat[][Adaptive-PINNs, $t=0.03$]{\includegraphics[width=.5\textwidth]{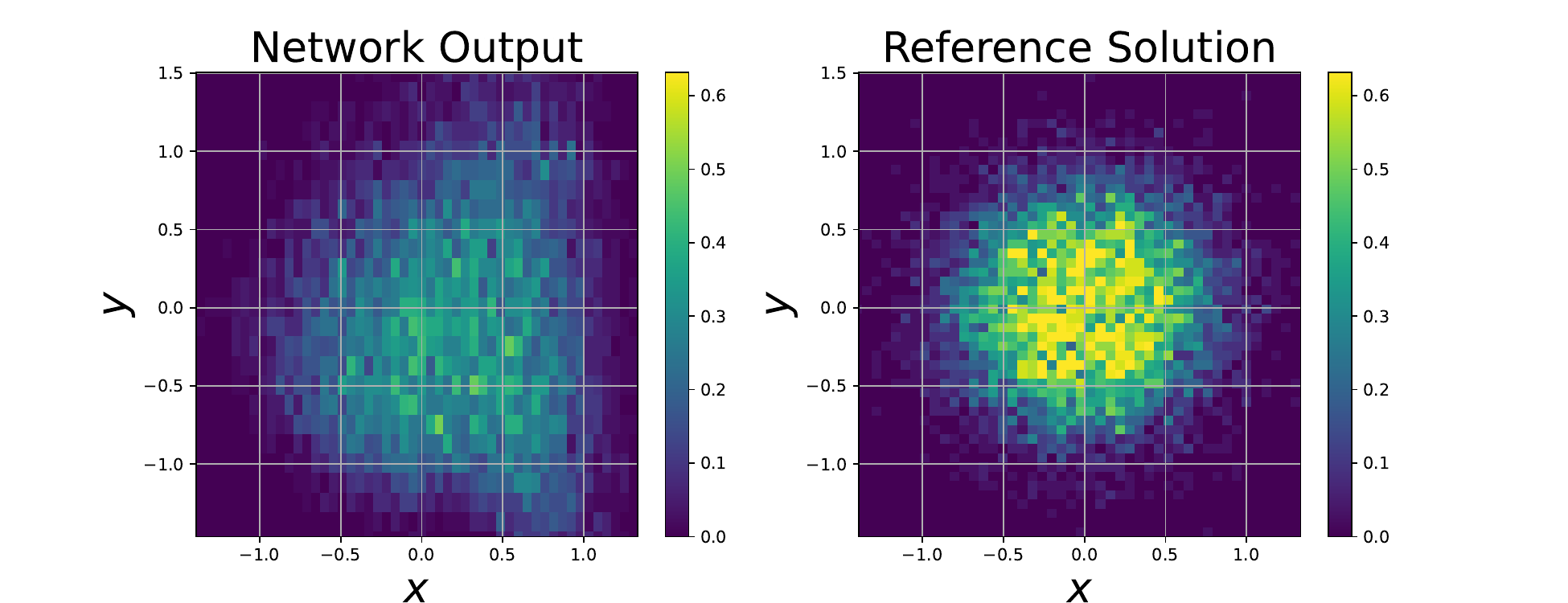}}
	\caption{Comparison of 2D KS solutions obtained by our DeepLagrangian method, PINNs and Adaptive-PINNs, test problem 1.}
    \label{problem1_compare}
\end{figure}

\subsection{Test problem 2: 2D KS system without advection for different initial conditions}
In this test problem, we also consider a 2D KS system without advection, but the initial conditions are different from test problem 1. Specifically, the initial conditions considered are composed of two distributions.

\subsubsection{Case 1: Two closely spaced distributions composing initial conditions}
In this case, the initial condition is given by 
\begin{align*}
	\rho^0(\mathbf{x})=7\pi p_1(\mathbf{x}) + 7\pi p_2(\mathbf{x}),
\end{align*}
where $p_1$ and $p_2$ are uniform distributions defined by
\begin{equation*}
	p_1(\mathbf{x})=p_1(x,y)=\left\{
	\begin{array}{ll}
	\frac{4}{\pi}, &(x-0.3)^2+(y-0.3)^2\leq 0.25,\\
	0, &(x-0.3)^2+(y-0.3)^2> 0.25,
	\end{array}
	\right.
\end{equation*}
\begin{equation*}
	p_2(\mathbf{x})=p_2(x,y)=\left\{
	\begin{array}{ll}
		\frac{4}{\pi}, &(x+0.5)^2+(y+0.5)^2\leq 0.25,\\
	0, &(x+0.5)^2+(y+0.5)^2> 0.25.
	\end{array}
	\right.
\end{equation*}
Let $\rho^{0,1}(\mathbf{x}):=7\pi p_1(\mathbf{x})$, $\rho^{0,2}(\mathbf{x}):=7\pi p_2(\mathbf{x})$.
Note that $\int_{\Omega} \rho^{0,1}(\mathbf{x}) d\mathbf{x}=\int_{\Omega} \rho^{0,2}(\mathbf{x}) d\mathbf{x}=7\pi \leq 8\pi$, so if $\rho^{0,1}(\mathbf{x})$ or $\rho^{0,2}(\mathbf{x})$ is set to the initial condition of the KS system, the system does not blow up.
In Case 1, we regard $\rho^{0,1}(\mathbf{x})+\rho^{0,2}(\mathbf{x})=\rho^0(\mathbf{x})$ as the initial condition and then the total mass $M=\int_{\Omega} \rho^0(\mathbf{x})d\mathbf{x}=14\pi \geq 8\pi$
will cause the system to blow up in finite time.
The final time is set to $T=0.09$. 

For conducting time marching strategies, $[0,T]$ is equally divided into three parts.
Our DeepLagrangian method (\Cref{alg}) constructs the time-dependent KRnet to approximate the normalized KS solutions. The prior of the time-dependent KRnet is set to a standard Gaussian distribution
$p_{\mathbf{z}}(\mathbf{z})$.
\Cref{problem2_noblow} provides the solutions of the 2D KS system without advection estimated by DeepLagrangian (Network Output) and IPM (Reference Solution), where it is evident that the two estimated solutions are consistent. From \Cref{problem2_noblow}(a) to \Cref{problem2_noblow}(d), it can be seen that two clusters of particles start to aggregate into a cluster of particles, and then the cluster blows up near $t=0.09$. 
\begin{figure}[!htb]
	\centering
	\subfloat[][$t=0.01$]{\includegraphics[width=.5
 \textwidth]{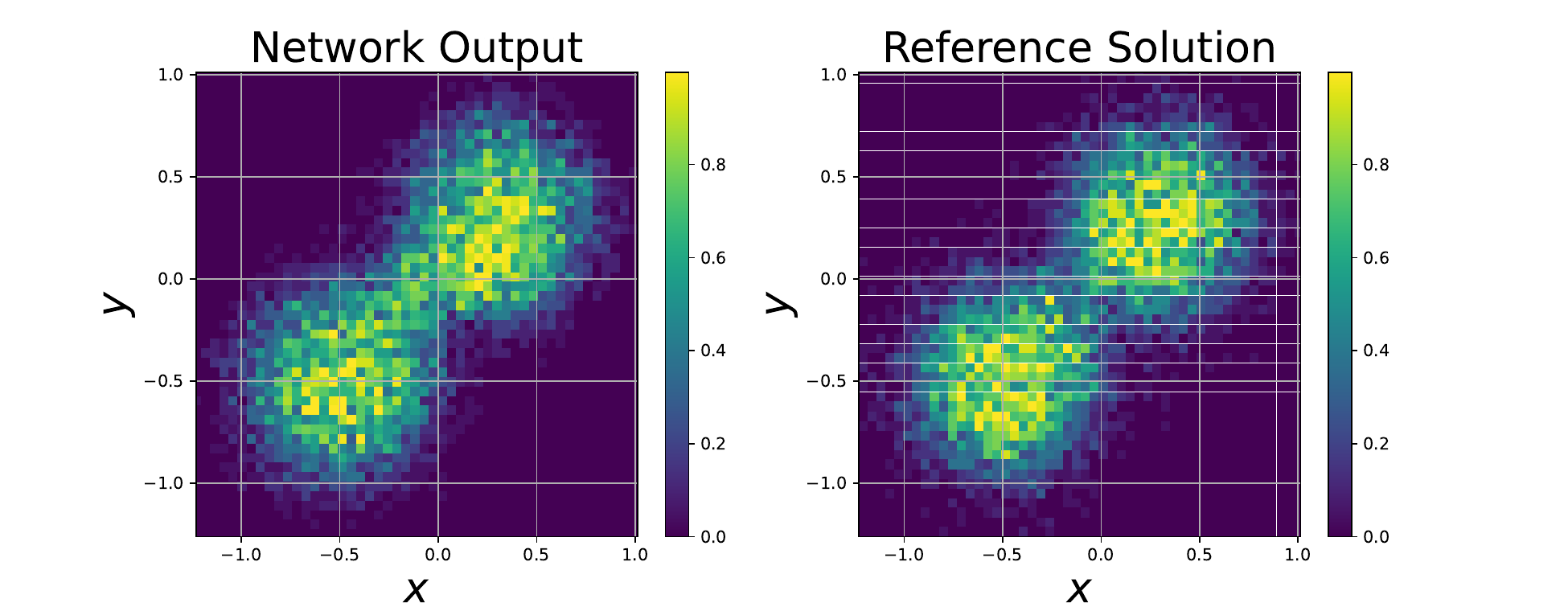}}
	\subfloat[][$t=0.04$]{\includegraphics[width=.5\textwidth]{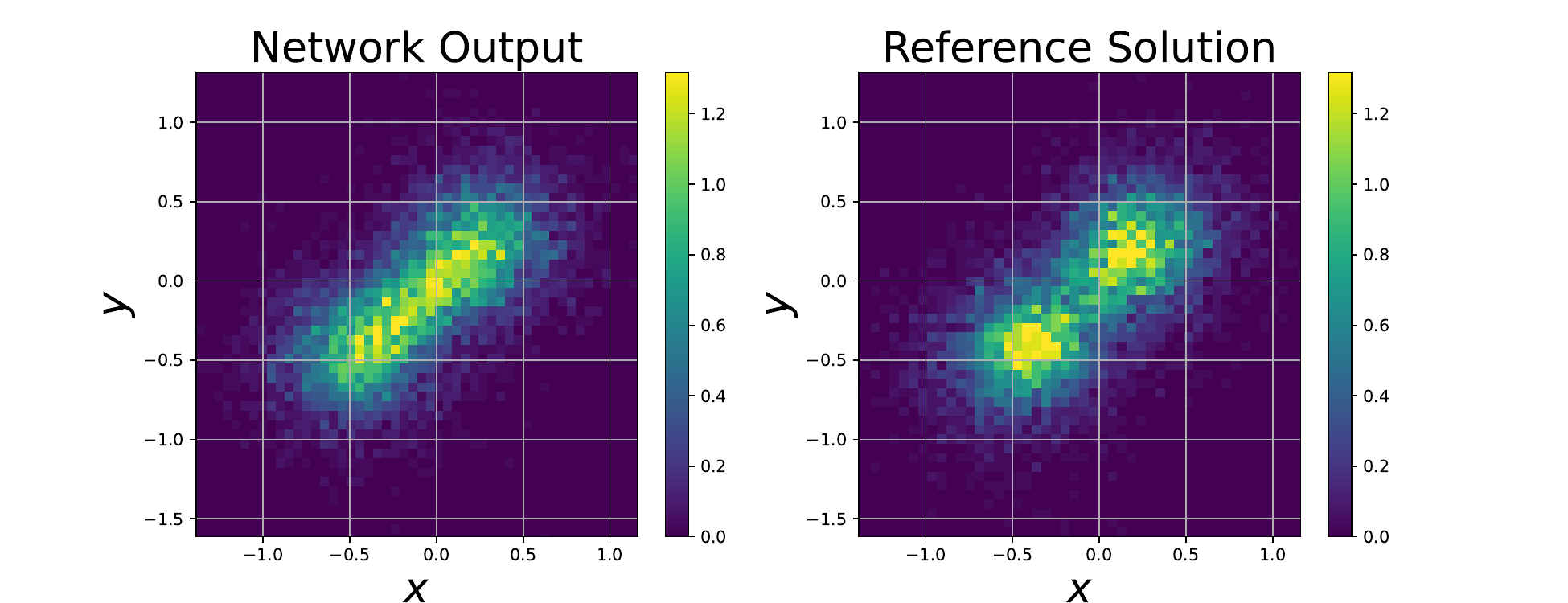}}\\
	\subfloat[][$t=0.08$]{\includegraphics[width=.5\textwidth]{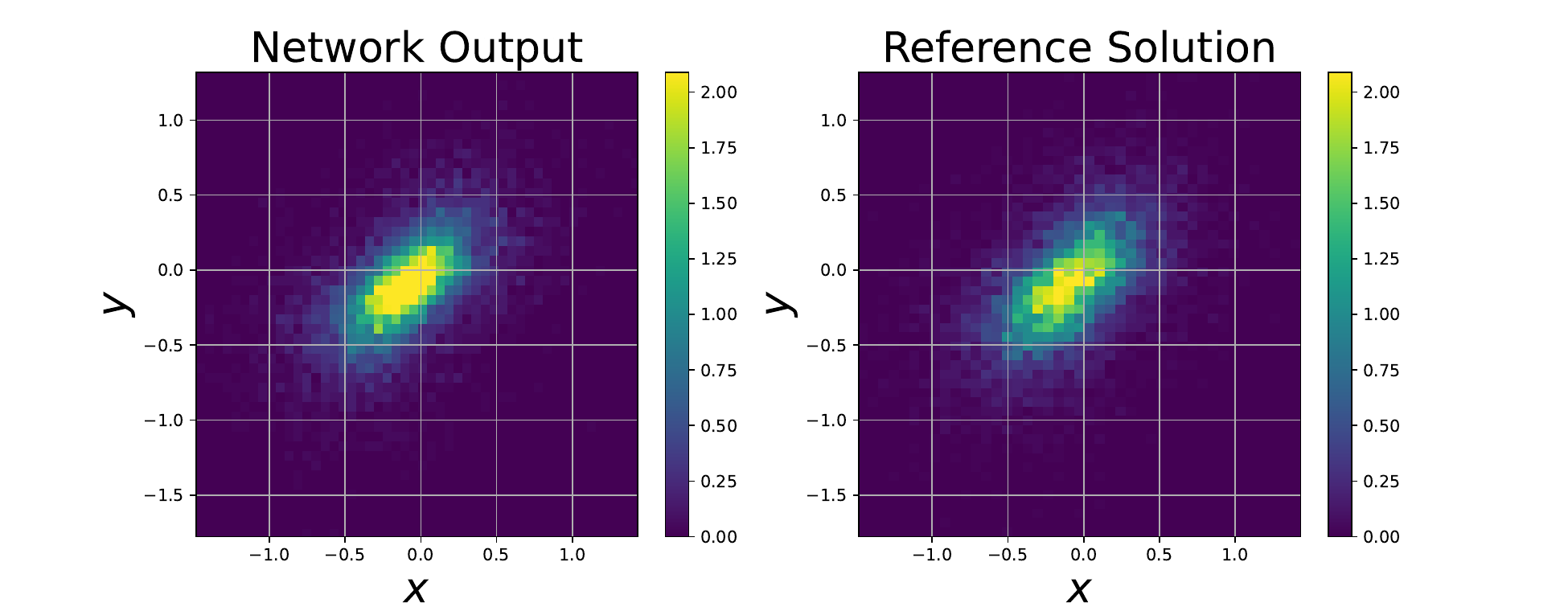}}
	\subfloat[][$t=0.09$ ]{\includegraphics[width=.5\textwidth]{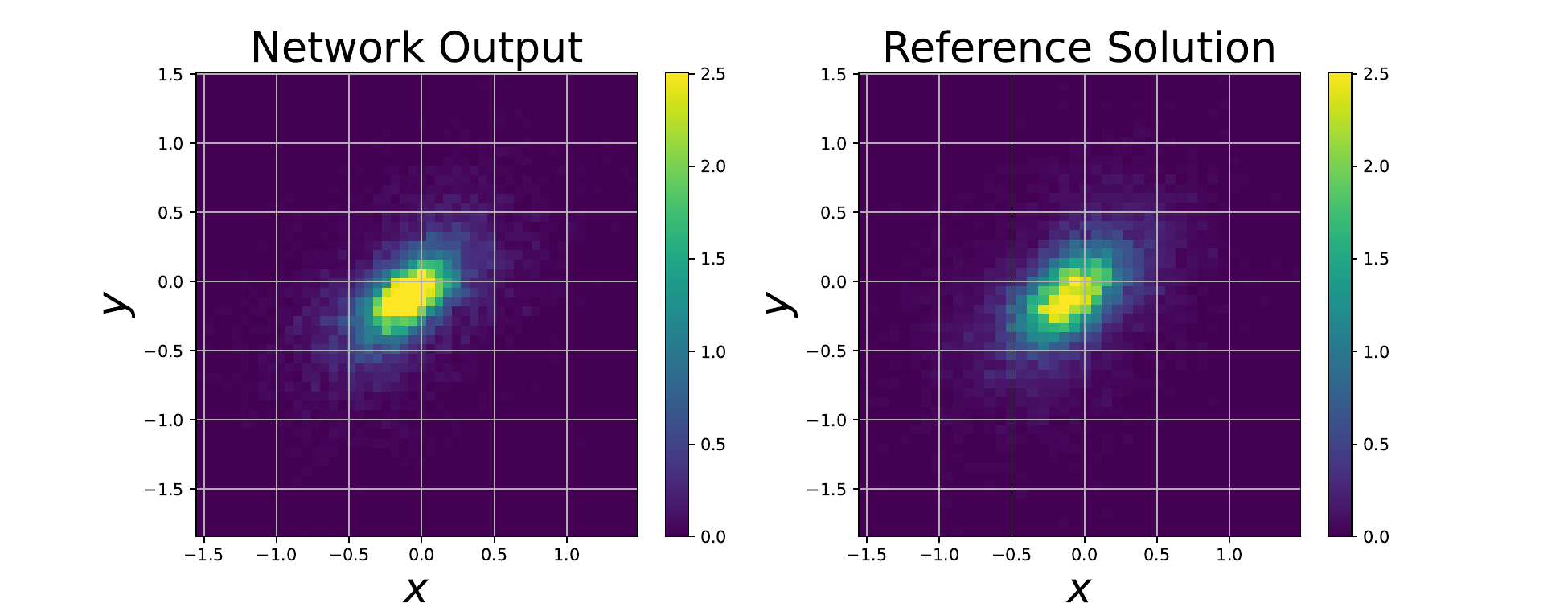}}
	\caption{Solutions of the 2D KS system without advection for Case 1, test problem 2.}
    \label{problem2_noblow}
\end{figure}
\subsubsection{Case 2: Two widely spaced distributions composing initial conditions}
In this case, the initial condition is set to be 
\begin{align*}
	\rho^0(\mathbf{x})=12\pi p_1(\mathbf{x})+ 12\pi p_2(\mathbf{x}),
\end{align*}
where $p_1$ and $p_2$ are uniform distributions defined by
\begin{equation*}
	p_1(\mathbf{x})=p_1(x,y)=\left\{
	\begin{array}{ll}
	\frac{1}{\pi}, &(x-1)^2+(y-1)^2\leq 1,\\
	0, &(x-1)^2+(y-1)^2> 1.
	\end{array}
	\right.
\end{equation*}
\begin{equation*}
	p_2(\mathbf{x})=p_2(x,y)=\left\{
	\begin{array}{ll}
		\frac{1}{\pi}, &(x+1)^2+(y+1)^2\leq 1,\\
	0, &(x+1)^2+(y+1)^2> 1.
	\end{array}
	\right.
\end{equation*}
The total mass is $M=24\pi$. The final time is set to $T=0.12$.

We equally decompose $[0,T]$ into three parts to achieve time marching strategies. Our DeepLagrangian method builds a time-dependent KRnet to approximate the normalized KS solutions. 
The prior $p_{\mathbf{z}}(\mathbf{z})$ of time-dependent KRnet is set to the following Gaussian mixture distribution,
\begin{equation*}
	p_{\mathbf{z}}(\mathbf{z})=\frac{1}{4\pi}\exp\Big(-\frac{1}{2}\Vert\mathbf{z}+3\Vert^2\Big)+\frac{1}{4\pi}\exp\Big(-\frac{1}{2}\Vert\mathbf{z}-3\Vert^2\Big).
\end{equation*} 
\Cref{problem2_mixed} illustrates the solutions of a 2D KS system without advection estimated by DeepLagrangian and IPM, where it is clear that the solution computed by DeepLagrangian aligns with the reference solution. From \Cref{problem2_mixed}(a) to \Cref{problem2_mixed}(d), we can see that two clusters of particles aggregate independently over time and finally blow up independently near $t=0.12$.  
\begin{figure}[!htb]
	\centering
	\subfloat[][$t=0.02$]{\includegraphics[width=.5
 \textwidth]{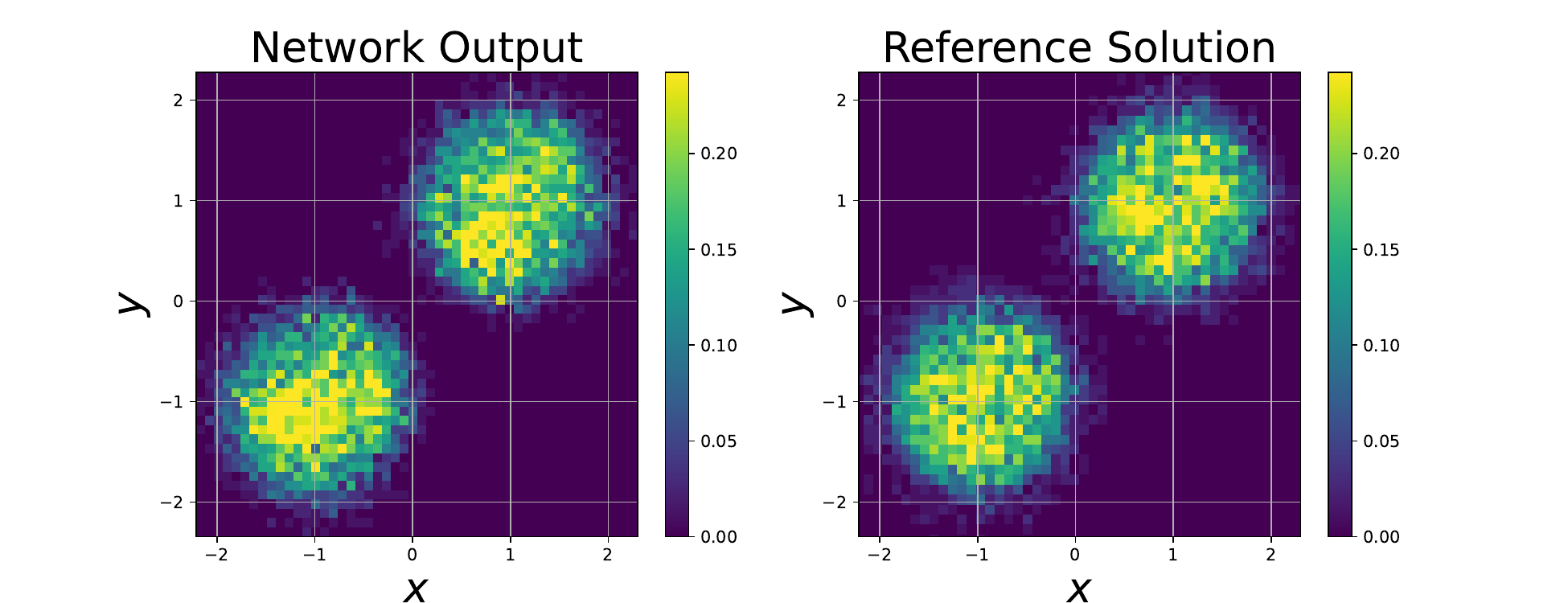}}
	\subfloat[][$t=0.06$]{\includegraphics[width=.5\textwidth]{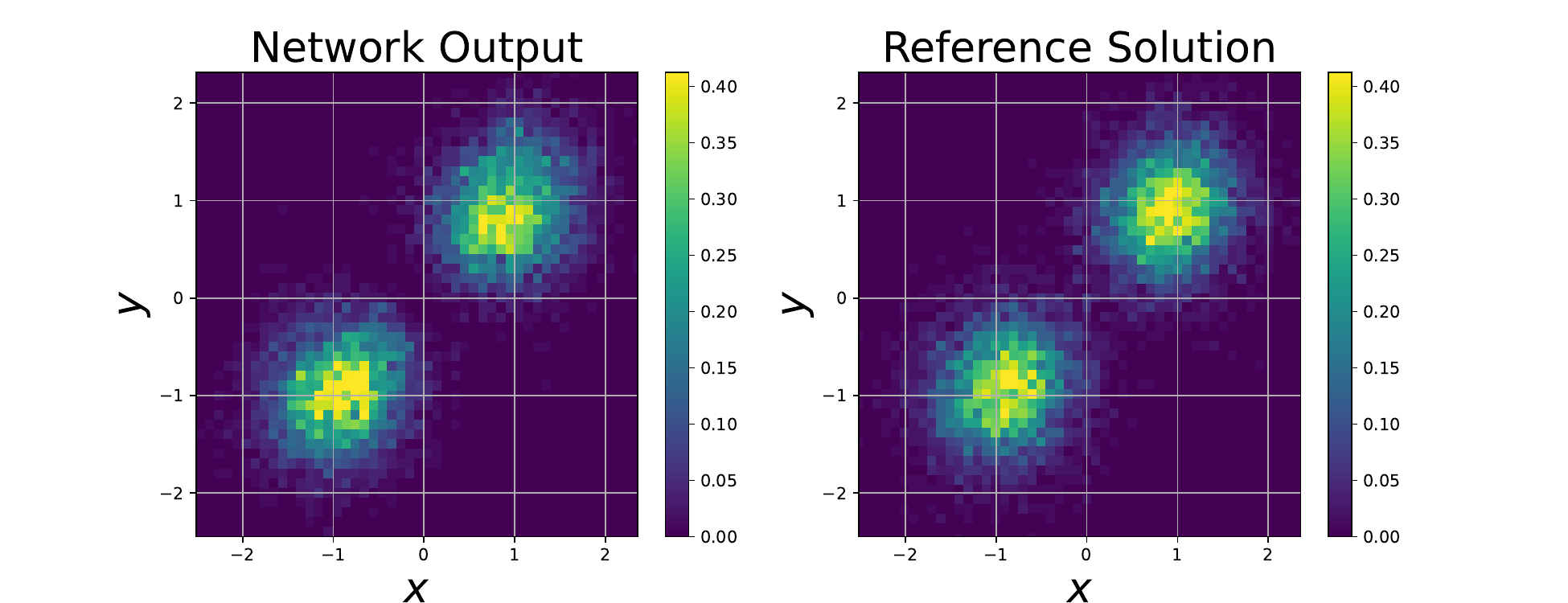}}\\
	\subfloat[][$t=0.08$]{\includegraphics[width=.5\textwidth]{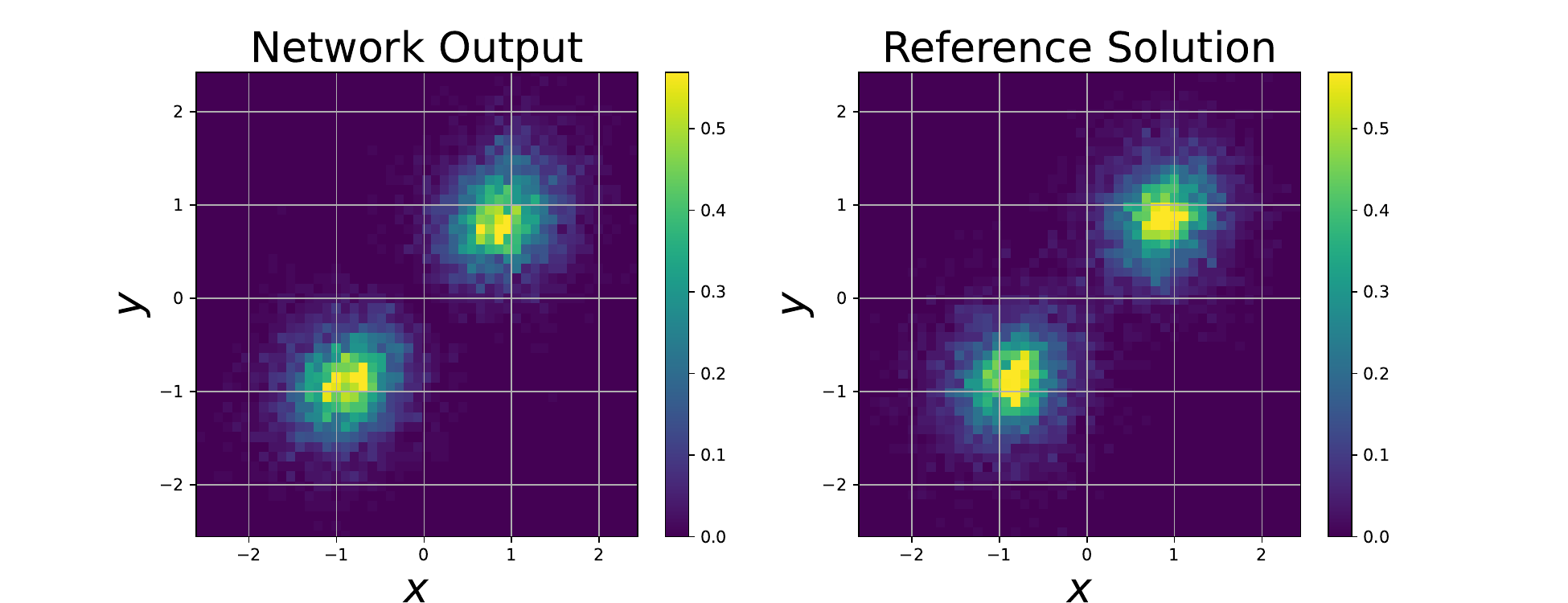}}
	\subfloat[][$t=0.12$]{\includegraphics[width=.5\textwidth]{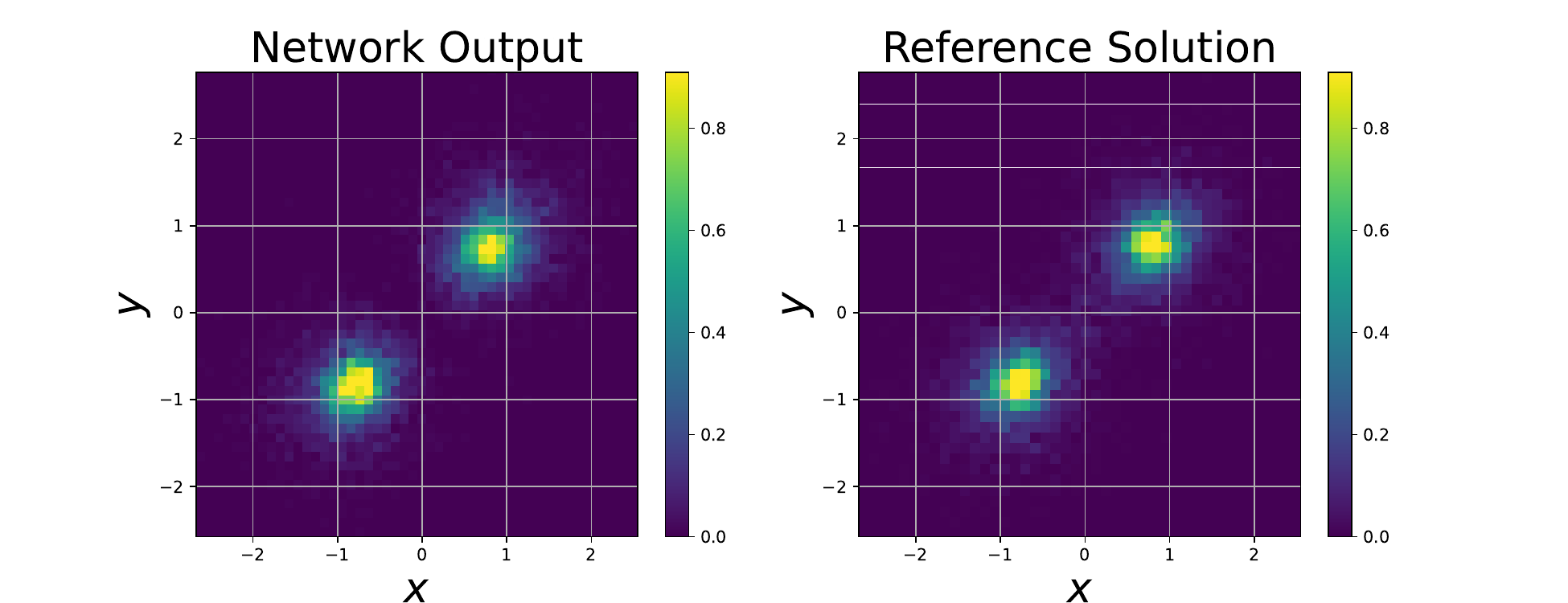}}
	\caption{Solutions of the 2D KS system without advection for Case 2, test problem 2.}
   \label{problem2_mixed}
\end{figure}

\subsection{Test problem 3: 2D KS system with advection }
In this test problem, we consider 
the 2D KS system \eqref{fks} with the following advection term,
\begin{align}\label{v_example}
	\mathbf{v}=\mathbf{v}(x,y)=A\left[\exp(-y^2) ,0\right]^T,
 \end{align}
 where $A$ represents the amplitude
 of the advection term and can also be regarded as a physical parameter.  The initial condition of the KS system is the same as that of test problem 1. Our goal is to study the dependence of the aggregation patterns of the 2D KS system on the evolution time $t$ and the physical parameter $A$ of the advection term.  
 
 \subsubsection{Dependence on the evolution time} \label{ev_part}
 To learn the dependence of 2D KS solutions on the evolution time, we fix the physical parameter $A=100$ of the advection term in \eqref{v_example} and set the final time $T=0.09$. The total mass is set to $M=16\pi$. 
 
 For considering time marching strategies, $[0,T]$ is equally divided into three parts. In each part, our DeepLagrangian method applies a time-dependent KRnet to approximate the solutions of a 2D KS system with advection, where a standard Gaussian distribution $p_{\mathbf{z}}(\mathbf{z})$ is assigned to the prior of the time-dependent KRnet.  \Cref{problem3_fixA} gives the solutions of the 2D KS system with advection at time $t=0.02,0.03,0.06,0.09$ estimated by our DeepLagrangian method and IPM, where the estimated solution of our method is in agreement with that of IPM.
 \begin{figure}[!htb]
	\centering
	\subfloat[][$t=0.02$]{\includegraphics[width=.5
 \textwidth]{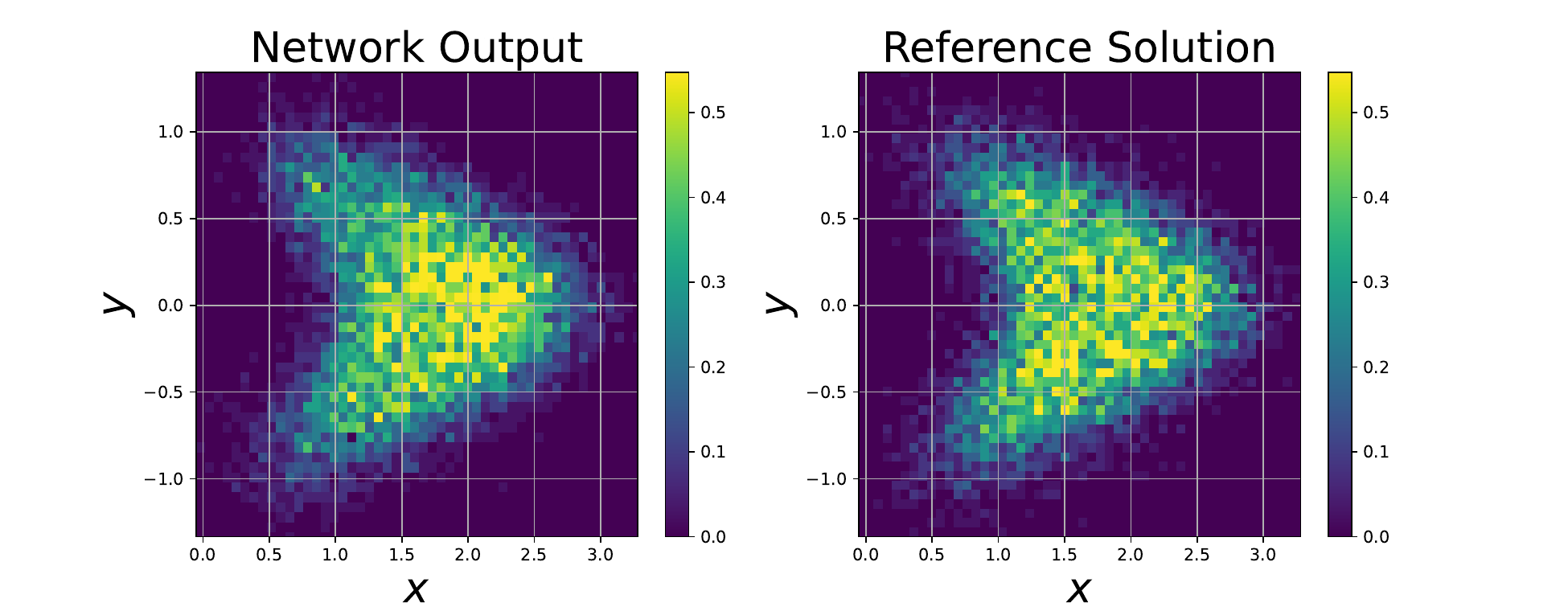}}
	\subfloat[][$t=0.03$]{\includegraphics[width=.5\textwidth]{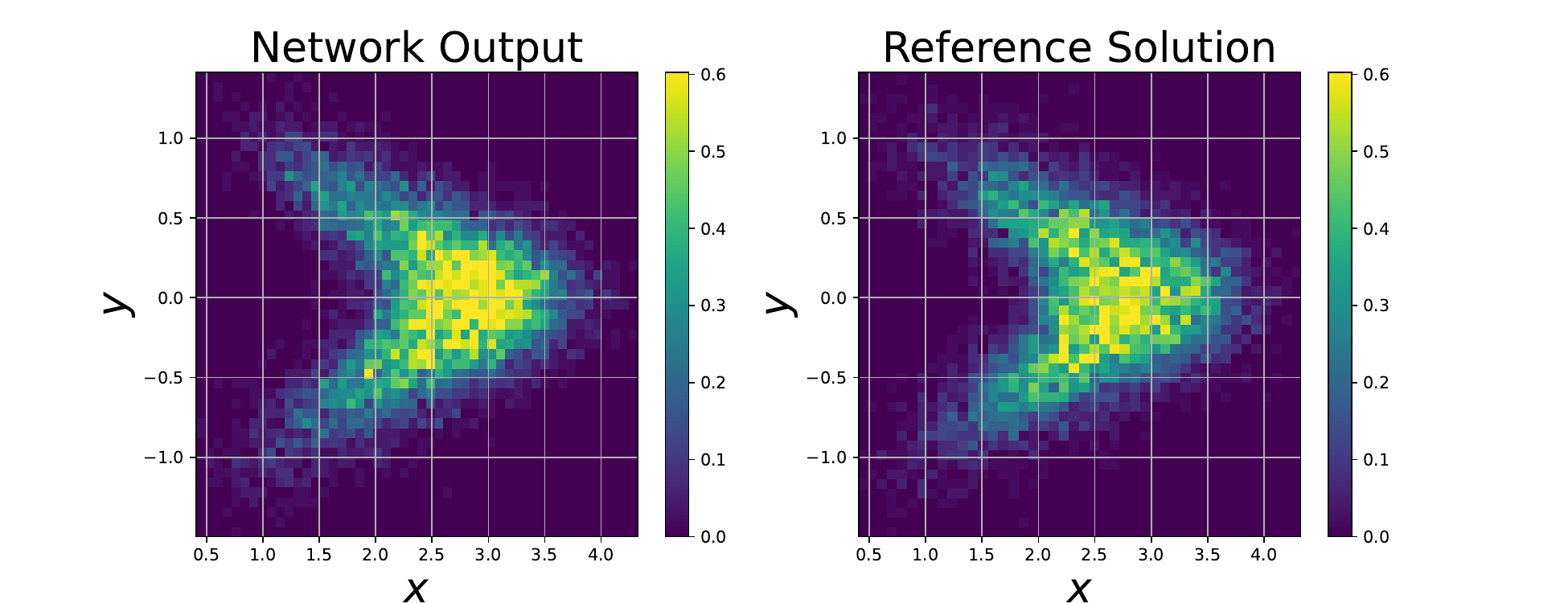}}\\
	\subfloat[][$t=0.06$]{\includegraphics[width=.5\textwidth]{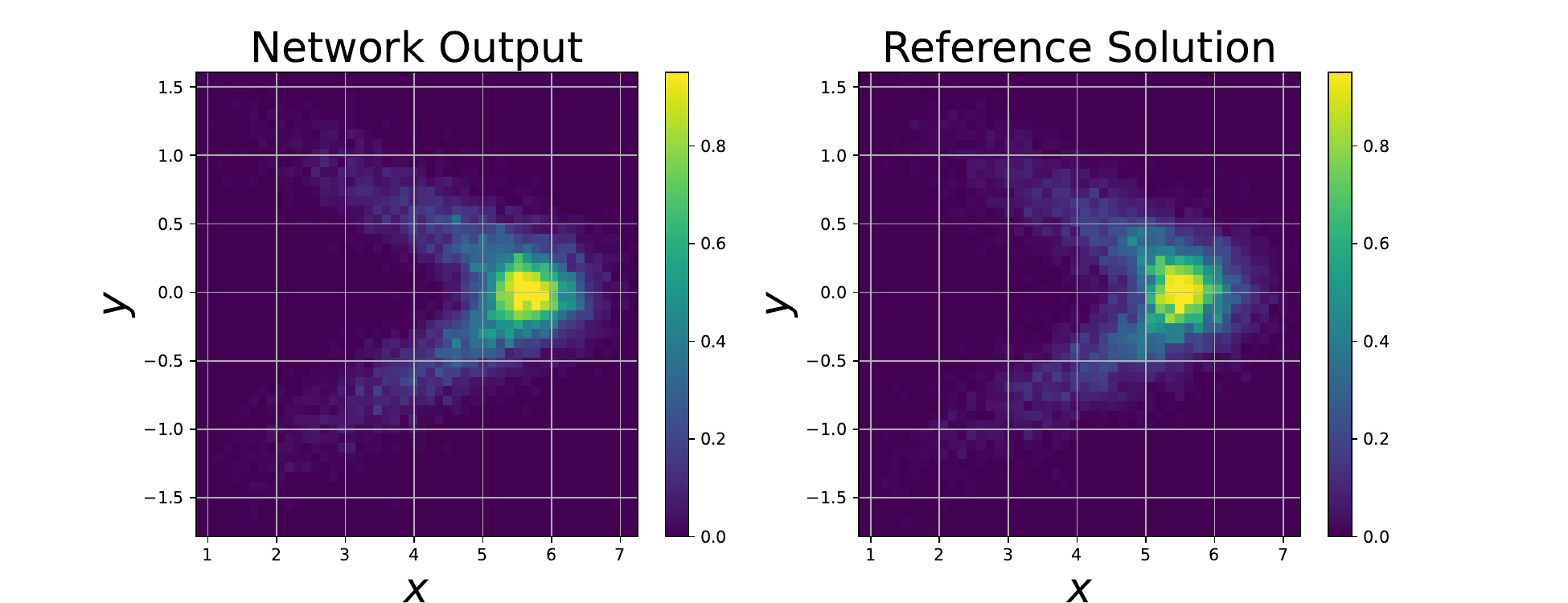}}
	\subfloat[][$t=0.09$]{\includegraphics[width=.5\textwidth]{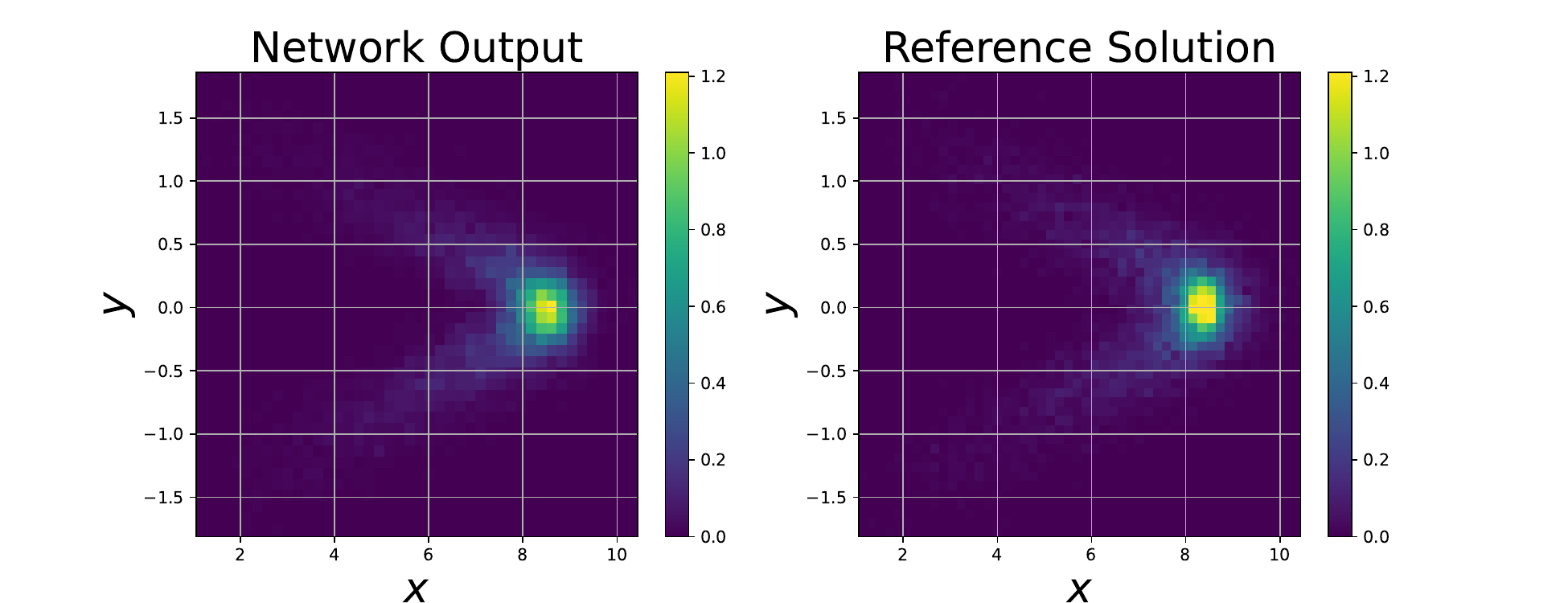}}
	\caption{Solutions of the 2D KS system with advection at different times, test problem 3.}
    \label{problem3_fixA}
\end{figure}
 \subsubsection{Dependence on the physical parameter}\label{section_parameter}
 To consider the dependence of 2D KS solutions on the physical parameter, we set the final time $T=0.02$ and the physical parameter $A=10^{0.2a}, a\sim \mathcal{U}(0,10)$.
 The total mass is set to $M=16\pi$.

Our method constructs a time-dependent KRnet to approximate the 2D KS solutions on the time interval $[0, T]$. The settings for the time-dependent KRnet are the same as \Cref{ev_part} except that the inputs of $\mathsf{NN}_1$ in \eqref{NN1} and $\mathsf{NN}_2$ in \eqref{NN2} are replaced by $(\mathbf{x}_{1},t,a)$ and $(\tilde{\mathbf{x}}_{2},t,a)$, respectively.
\Cref{problem3_fixT} shows 2D KS solutions under different parameters $A$ in \eqref{v_example} obtained by our DeepLagrangian method and IPM, which implies that the solution computed by our method aligns with the reference solution computed by IPM. 

 \begin{figure}[!htb]
	\centering
	\subfloat[][$A=10$]{\includegraphics[width=.5
 \textwidth]{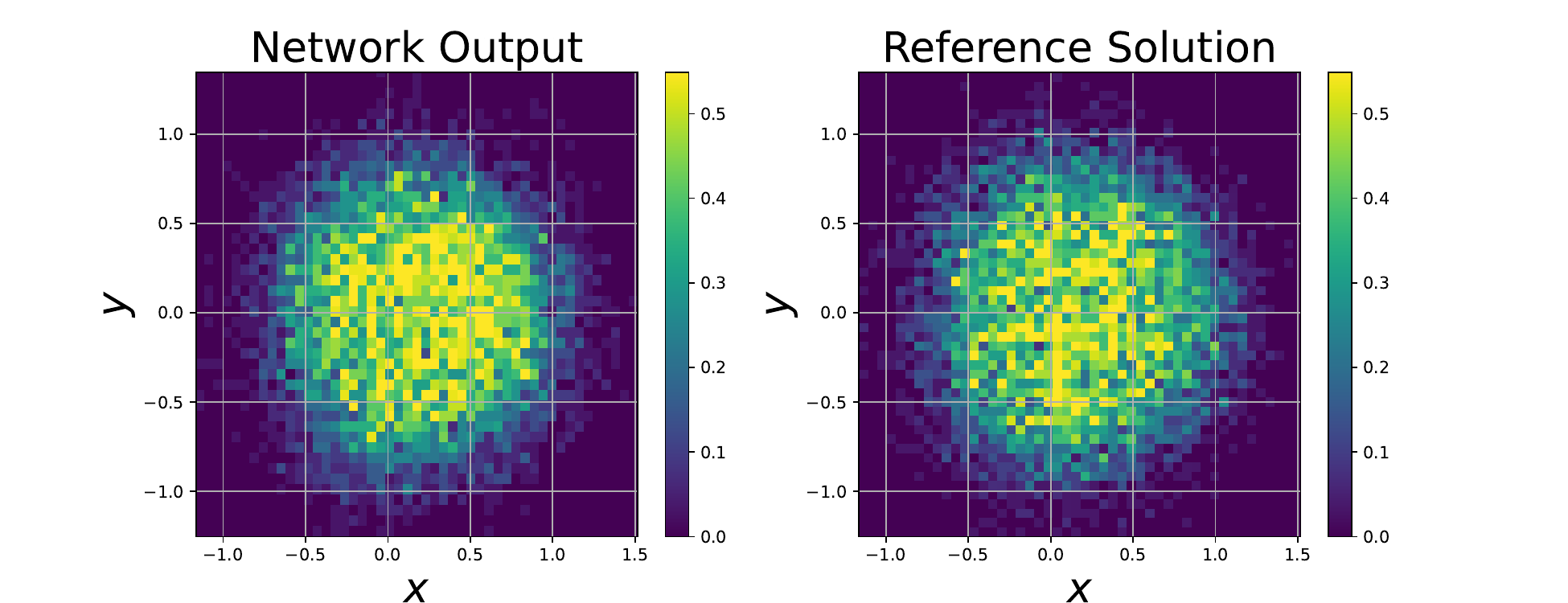}}
	\subfloat[][$A=80$]{\includegraphics[width=.5\textwidth]{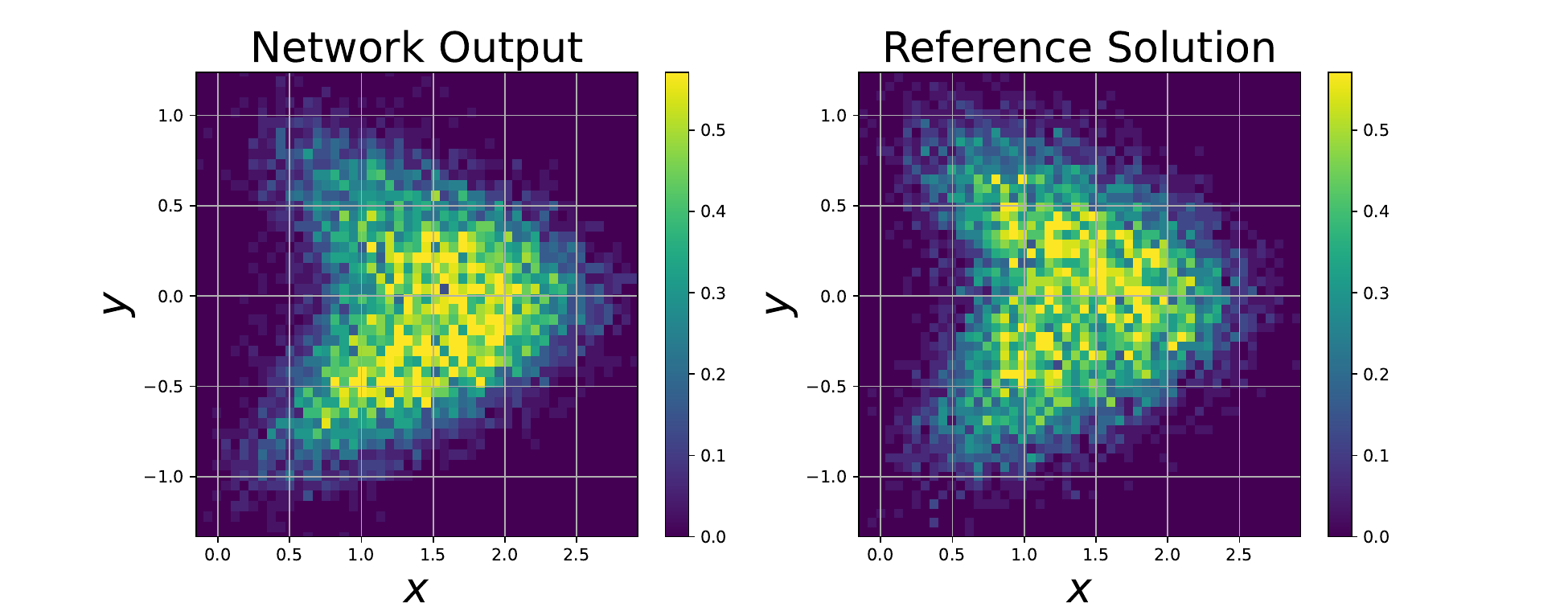}}\\
	\subfloat[][$A=100$]{\includegraphics[width=.5\textwidth]{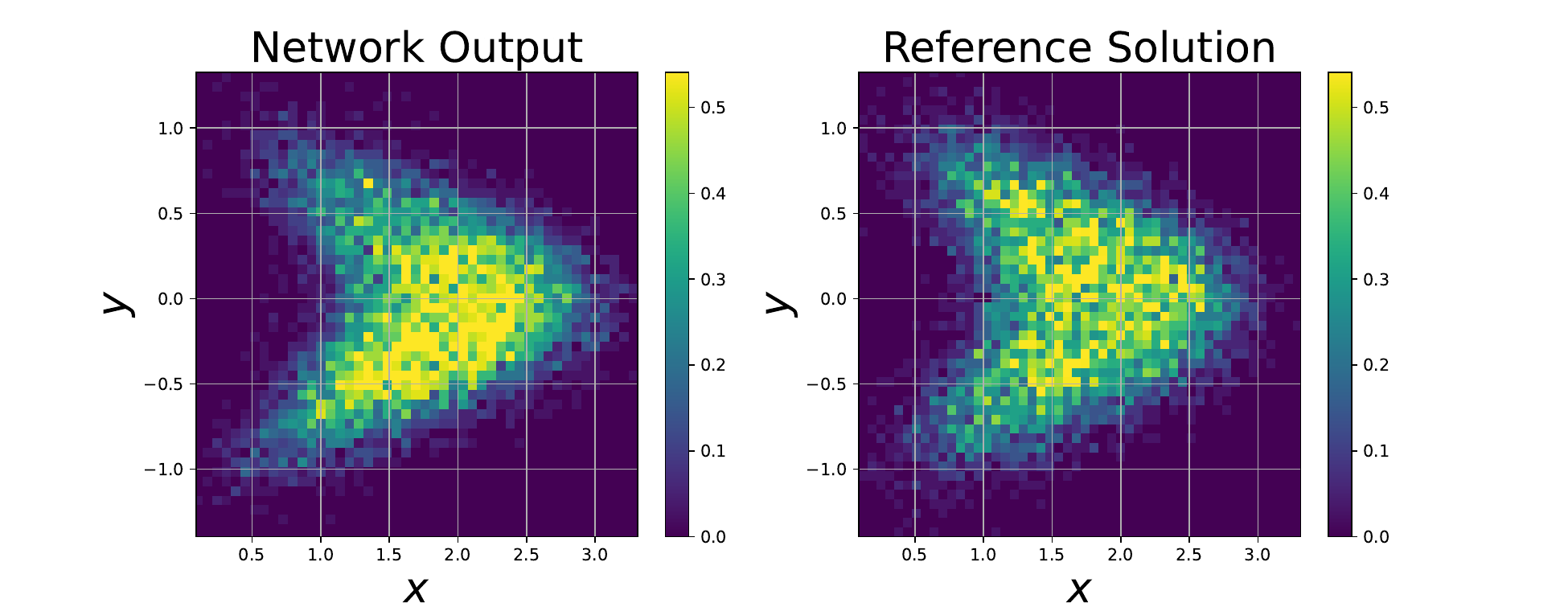}}
	\subfloat[][$A=150$ (prediction)]{\includegraphics[width=.5\textwidth]{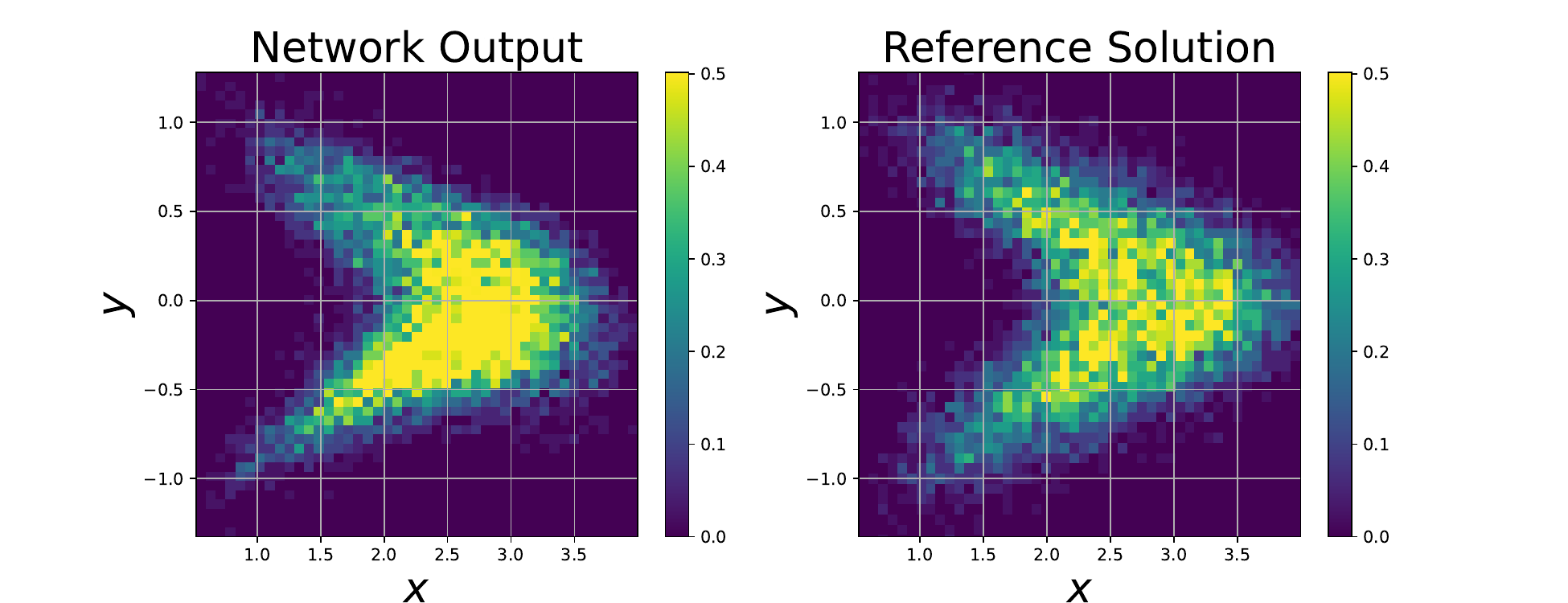}}
	\caption{Solutions of 2D KS system with advection for different parameters, test problem 3.}
    \label{problem3_fixT}
\end{figure}

\subsection{Test problem 4: 3D KS system}\label{test4_3D}
The KS system with $d=3$ and $\mathbf{v}=0$ in \eqref{fks} is considered. The initial condition is set to an unnormalized uniform distribution over a 3D ball of radius 1 centered at the origin.  The total mass is set to $M=24\pi$. The final time is set to $T=0.09$.

To achieve time marching strategies, $[0,T]$ is equally divided into three parts and then our DeepLagrangian method applies a time-dependent KRnet to approximate the KS solution on each part.
For the time-dependent KRnet, its prior $p_{\mathbf{z}}(\mathbf{z})$ is specified as a standard Gaussian distribution.
\Cref{problem4_plot} gives the histogram of the network output of our method projected on the $xy$ plane, which is
consistent with that of IPM. 
\Cref{problem4_plot_3d} provides the particle distributions generated by our method and IPM at $t=0.01,0.04,0.08,0.09$, where it is clear that the particles aggregate over time and the particle distribution obtained by our method looks very similar to that obtained by IPM. 
\begin{figure}[!htb]
	\centering
	\subfloat[][$t=0.01$]{\includegraphics[width=.5
 \textwidth]{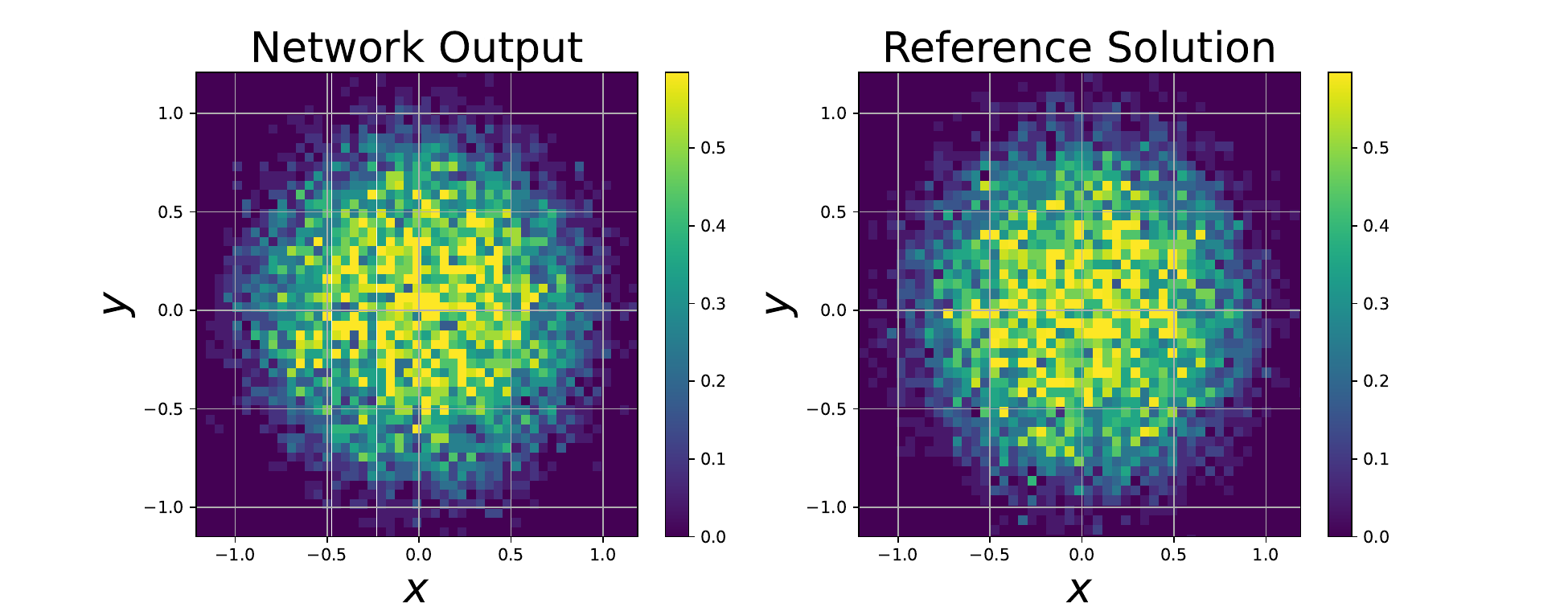}}
	\subfloat[][$t=0.04$]{\includegraphics[width=.5\textwidth]{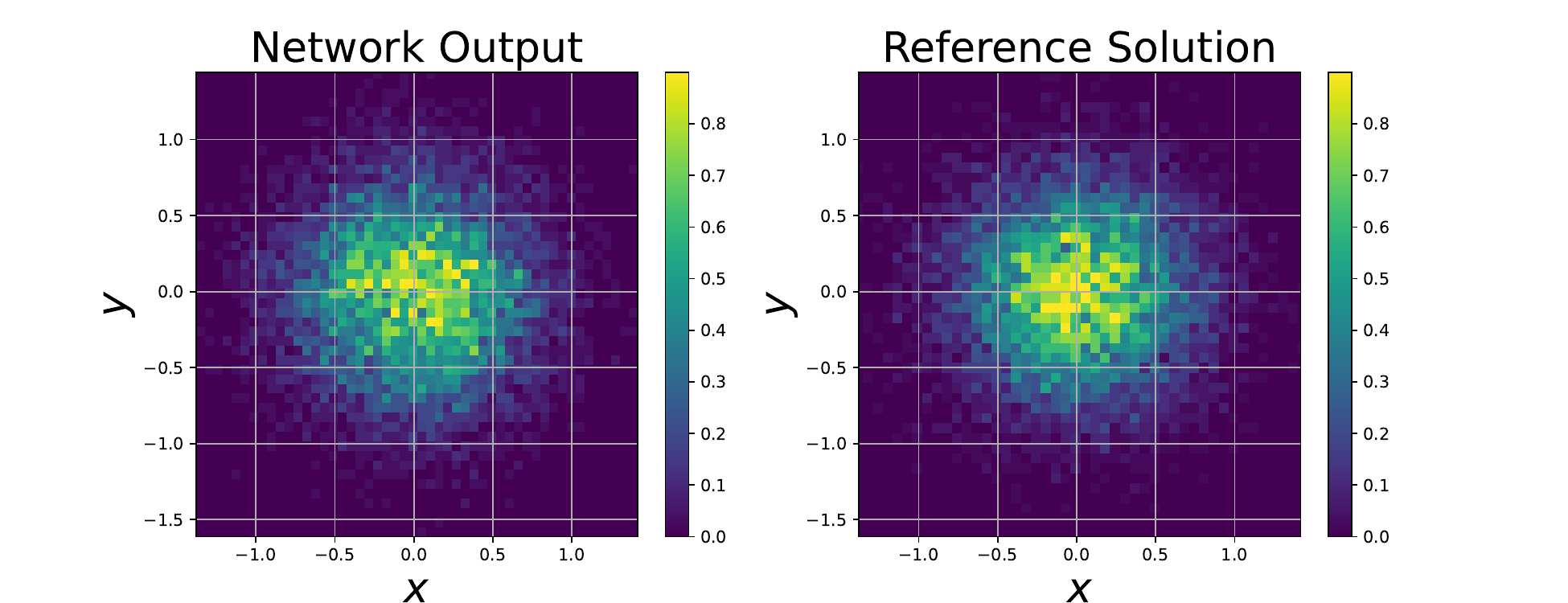}}\\
	\subfloat[][$t=0.08$]{\includegraphics[width=.5\textwidth]{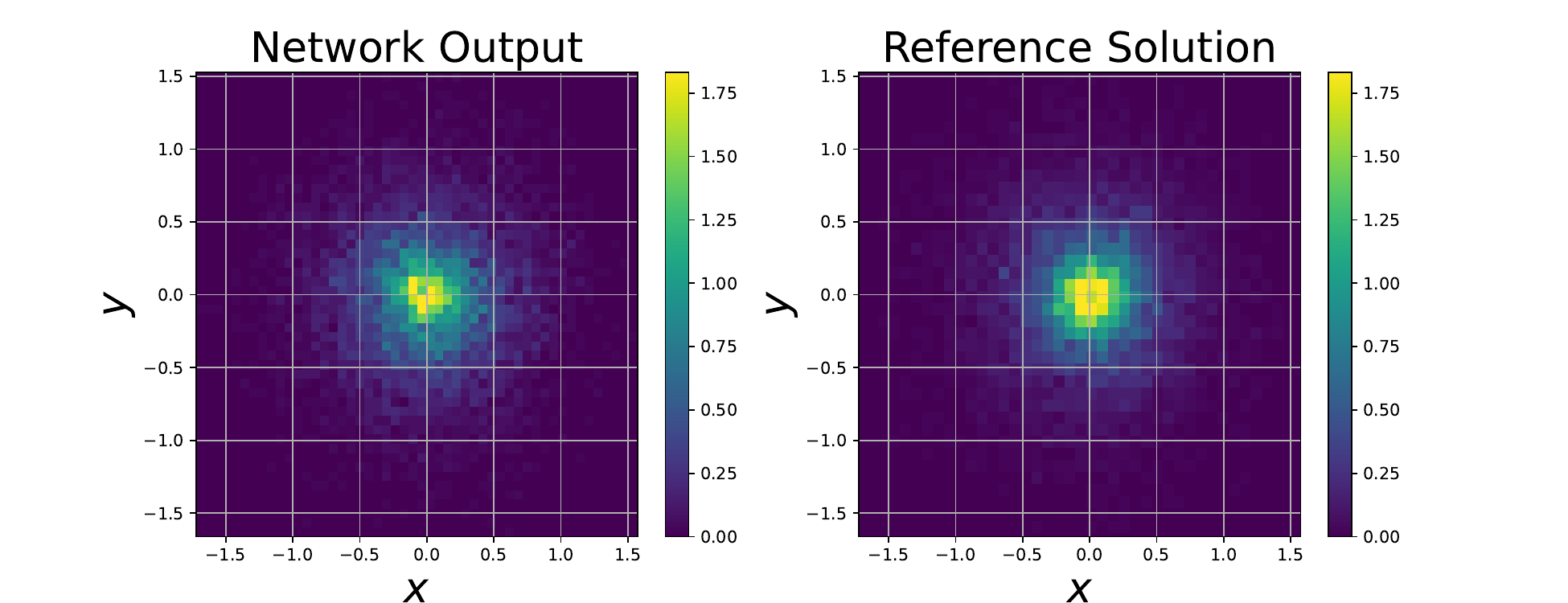}}
	\subfloat[][$t=0.09$]{\includegraphics[width=.5\textwidth]{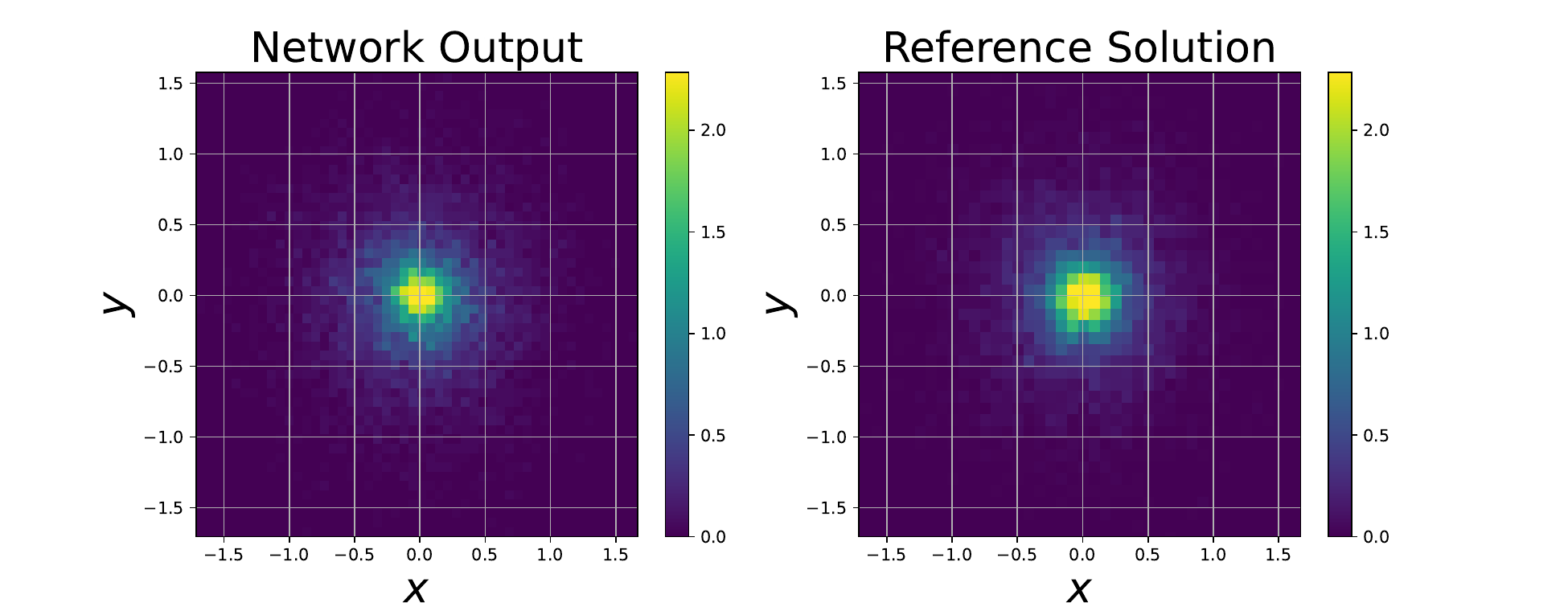}}
	\caption{Solutions of the 3D KS system projected to $xy$ plane, test problem 4.}
    \label{problem4_plot}
\end{figure}
\begin{figure}[H]
    \centering
\includegraphics[width=0.95\linewidth]{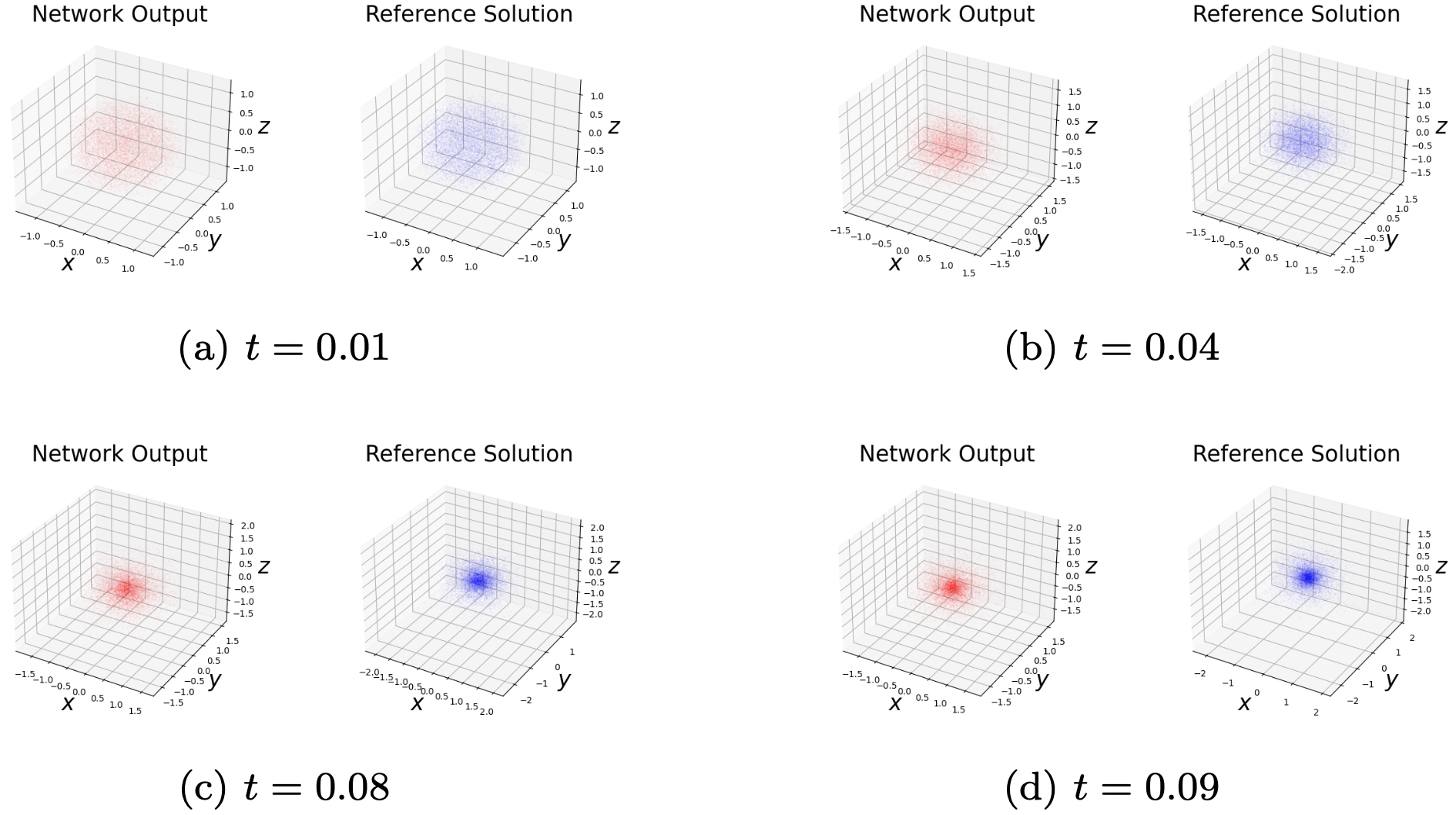}
    \caption{Particle distributions of the 3D KS system, test problem 4.}
    \label{problem4_plot_3d}
\end{figure}

\section{Conclusion}
This paper presents a DeepLagrangian method for learning aggregation patterns of 
Keller-Segel (KS) chemotaxis systems, particularly in handling near-singular solutions such as blow-up and concentration phenomena. By normalizing KS solutions into a probability density function (PDF) and then reformulating the KS system into a Lagrangian framework, our method can adapt to the near-singular solutions. We define a physics-informed Lagrangian loss for enforcing the Lagrangian framework, and when combined with time-marching strategies, we employ a time-dependent KRnet to approximate the PDF by minimizing the loss. This allows us to accurately approximate the solutions of 2D and 3D KS systems under different initial conditions and physical parameters. Compared to our method without time-marching strategies, as well as PINNs and Adaptive-PINNs, our method with time-marching strategies demonstrates superior accuracy in capturing complex aggregation patterns and near-singular behaviors. Future work will focus on developing adaptive time-marching strategies to achieve adaptive time-interval lengths and extending the proposed method to other complicated KS systems, such as those describing cancer cell invasion \cite{hu2024stochastic}.

\section*{Acknowledgments} 
The simulations are performed using research computing facilities offered by Information Technology Services, The University of Hong Kong.

\bibliographystyle{siamplain}
\bibliography{main}
\end{document}